\documentclass{amsart}

\usepackage{amsmath, amssymb, amscd, mathrsfs, bm}
\usepackage{xcolor}
\usepackage[linktocpage=true]{hyperref}
\usepackage[only,mapsfrom]{stmaryrd}

\newtheorem{theorem}{Theorem}[section]
\newtheorem{proposition}[theorem]{Proposition}

\newtheorem{definition}[theorem]{Definition}

\theoremstyle{remark}
\newtheorem{remark}[theorem]{Remark}

\newtheorem*{acknowledgments}{Acknowledgments}

\numberwithin{equation}{section}

\newcommand{\Order}{\mathcal{O}}
\newcommand{\into}{\hookrightarrow}

\newcommand{\onto}{\twoheadrightarrow}
\newcommand{\isomto}{\overset{\sim}{\to}}

\newcommand{\tensor}{\mathbin{\otimes}}
\newcommand{\closure}[1]{\overline{#1}}

\newcommand{\Z}{\mathbb{Z}}
\newcommand{\Q}{\mathbb{Q}}

\newcommand{\F}{\mathbb{F}}
\newcommand{\et}{\mathrm{et}}
\newcommand{\fppf}{\mathrm{fppf}}

\newcommand{\Ga}{\mathbf{G}_{a}}

\newcommand{\dirlim}{\varinjlim}
\newcommand{\invlim}{\varprojlim}
\newcommand{\sep}{\mathrm{sep}}
\newcommand{\ur}{\mathrm{ur}}
\newcommand{\ind}{\mathrm{ind}}
\newcommand{\rat}{\mathrm{rat}}
\newcommand{\pro}{\mathrm{pro}}
\newcommand{\Alg}{\mathrm{Alg}}
\newcommand{\alg}[1]{\mathbf{#1}}
\newcommand{\var}{\;\cdot\;}
\newcommand{\ideal}[1]{\mathfrak{#1}}
\newcommand{\ctensor}{\mathbin{\Hat{\otimes}}}

\newcommand{\Loc}{\mathrm{L}}
\newcommand{\uc}{\mathrm{uc}}
\newcommand{\ca}{\mathrm{ca}}
\newcommand{\Pro}{\mathrm{P}}
\newcommand{\Ind}{\mathrm{I}}
\newcommand{\Et}{\mathrm{Et}}
\newcommand{\SDual}{\mathrm{SD}}
\DeclareMathOperator{\Gal}{Gal}

\DeclareMathOperator{\Spec}{Spec}
\DeclareMathOperator{\Ab}{Ab}

\DeclareMathOperator{\Lie}{Lie}

\DeclareMathOperator{\length}{length}
\DeclareMathOperator{\Res}{Res}

\DeclareMathOperator{\Gr}{Gr}

\DeclareMathOperator{\Sel}{Sel}
\DeclareMathOperator{\sheafhom}{\mathbf{Hom}}

\DeclareFontFamily{U}{wncy}{}
\DeclareFontShape{U}{wncy}{m}{n}{<->wncyr10}{}
\DeclareSymbolFont{mcy}{U}{wncy}{m}{n}
\DeclareMathSymbol{\Sha}{\mathord}{mcy}{"58}

\makeatletter
\@namedef{subjclassname@2020}{\textup{2020} Mathematics Subject Classification}
\makeatother

\title[$\mu$-invariants and isogenies]
{On $\mu$-invariants and isogenies for abelian varieties over function fields}

\author[Ghosh]{Sohan Ghosh}
\address[Ghosh]{
	Indian Institute of Science Bangalore, Bengaluru 560012, India
}
\email{ghoshsohan4@gmail.com}

\author[Ray]{Jishnu Ray\(^{1,2}\)}
\address[Ray]{\!\!\(^1\) Harish-Chandra Research Institute, Chhatnag Road, Jhunsi, Prayagraj 211 019, India.}
\address[Ray]{\!\!\(^2\) Homi Bhabha National Institute, Training School Complex, Anushakti Nagar, Mumbai 400 094, India.}
\email{jishnuray@hri.res.in; jishnuray1992@gmail.com}

\author[Suzuki]{Takashi Suzuki}
\address[Suzuki]{
	Department of Mathematics, Chuo University,
	1-13-27 Kasuga, Bunkyo-ku, Tokyo 112-8551, Japan
}
\email{tsuzuki@gug.math.chuo-u.ac.jp}

\date{June 25, 2026}
\subjclass[2020]{11R23 (Primary) 11G10, 11R58, 11G40, 14F20 (Secondary)}
\keywords{Iwasawa invariants; abelian varieties;
function fields; Birch--Swinnerton-Dyer conjectures}

\begin{document}

\begin{abstract}
	We give several formulas for
	how Iwasawa $\mu$-invariants of abelian varieties
	over unramified $\Z_{p}$-extensions of function fields change under isogeny.
	These are analogues of Schneider's formula in the number field setting.
	We also prove that the validity of the Birch--Swinnerton-Dyer conjecture
	(including the leading coefficient formula) over function fields
	is invariant under isogeny,
	without using the result of Kato--Trihan.
\end{abstract}

\maketitle

\tableofcontents


\section{Introduction}
\label{0079}

Schneider \cite{Sch87} studies
how Iwasawa $\mu$-invariants of dual Selmer groups of abelian varieties
over cyclotomic $\Z_{p}$-extensions of number fields change
under isogenies of abelian varieties.
On the other hand, Iwasawa theory for abelian varieties over
unramified $\Z_{p}$-extensions of global function fields in one variable
of characteristic $p$ is initiated by Ochiai--Trihan \cite{OT09}.
In this paper, we will carry out an analogue of Schneider's studies
over unramified $\Z_{p}$-extensions of global function fields.

Let $f \colon A_{1} \to A_{2}$ be an isogeny of abelian varieties
over a global function field $K$ of characteristic $p > 0$.
Let $k$ be the constant field of $K$.
Let $q = \# k$.
Let $K_{\infty}$ be the (unique) unramified $\Z_{p}$-extension of $K$,
with subextension $K_{n}$ of degree $p^{n}$ for each $n \ge 0$.
Set $\Lambda = \Z_{p}[[\Gal(K_{\infty} / K)]]$.
For $i = 1, 2$ and $n \ge 0$, let $\Sel(A_{i} / K_{n})$ be
the Selmer group of $A_{i} \times_{K} K_{n}$.
Set $\Sel(A_{i} / K_{\infty}) = \dirlim_{n} \Sel(A_{i} / K_{n})$,
which has a natural action by $\Gal(K_{\infty} / K)$.
The Pontryagin dual $\Sel(A_{i} / K_{\infty})[p^{\infty}]^{\vee}$
of the $p$-primary part of $\Sel(A_{i} / K_{\infty})$
is a finitely generated torsion $\Lambda$-module
by \cite[Theorem 1.7]{OT09}.
Define the $\mu$-invariant $\mu_{A_{i} / K}$ to be
$1 / [k : \F_{p}]$ times the length of 
	\[
			\Sel(A_{i} / K_{\infty})[p^{\infty}]^{\vee}
		\tensor_{\Lambda}
			\Lambda_{(p)}
	\]
as a $\Lambda_{(p)}$-module.
We will give three different formulas for $\mu_{A_{2} / K} - \mu_{A_{1} / K}$
in terms of the kernel of $f$ and some local quantities.

Let $X$ be the proper smooth curve over $k$ with function field $K$.
Let $\mathcal{A}_{i}$ be the N\'eron model over $X$ of $A_{i}$.
Let $\mathcal{N}_{f} \subset \mathcal{A}_{1}$ be the schematic closure
of the kernel of $f \colon A_{1} \to A_{2}$ in $\mathcal{A}_{1}$,
which is a separated, quasi-finite, flat group scheme over $X$.
In Definition \ref{0048}, we will define a certain nonnegative integer $c(f)$,
which we call the \emph{conductor} of $f$.
It is a locally defined quantity
that measures the potential failure of $f$ to extend to an isogeny on the N\'eron models.

For a nonzero rational number $r$, let $v_{q}(r)$ be the $q$-valuation of $r$,
which is the $p$-adic valuation divided by $[k : \F_{p}]$
(so $v_{q}(q) = 1$).
Our first formula is the following:

\begin{theorem}[Proposition \ref{0049}, Theorem \ref{0004}] \label{0050} \mbox{}
	\begin{enumerate}
		\item
			For any separated, quasi-finite, flat group scheme $\mathcal{N}$ over $X$,
			the flat cohomology group $H^{i}(X, \mathcal{N})$ is finite for all $i$
			and zero for almost all $i$.
			In particular, we can define a number $\chi(X, \mathcal{N})$ by
				\[
						\chi(X, \mathcal{N})
					:=
						\prod_{i}
							(\# H^{i}(X, \mathcal{N}))^{(-1)^{i}}.
				\]
		\item
			We have
				\[
						\mu_{A_{2} / K} - \mu_{A_{1} / K}
					=
							v_{q}(\chi(X, \mathcal{N}_{f}))
						-
							c(f).
				\]
			In particular, if $f$ extends to an isogeny on the N\'eron models, then
				\[
						\mu_{A_{2} / K} - \mu_{A_{1} / K}
					=
						v_{q}(\chi(X, \mathcal{N}_{f})).
				\]
	\end{enumerate}
\end{theorem}

We can rewrite the term $v_{q}(\chi(X, \mathcal{N}_{f}))$
with a certain linear or coherent object.
For any separated, quasi-finite, flat group scheme $\mathcal{N}$ over $X$,
let $R \Lie \mathcal{N} := l_{\mathcal{N}}^{\vee}$ be the Lie complex
as defined by Illusie \cite[Chapter VII, Section 3.1.1]{Ill72}.
It is a perfect complex of $\Order_{X}$-modules
whose zeroth cohomology is the usual Lie algebra $\Lie \mathcal{N}$.
Let $\deg(R \Lie \mathcal{N})$ be its degree with respect to the base field $k$.

\begin{theorem}[Theorem \ref{0030}] \label{0053}
	For any separated, quasi-finite, flat group scheme $\mathcal{N}$ over $X$,
	we have
		\[
				\chi(X, \mathcal{N})
			=
				q^{\deg(R \Lie \mathcal{N})}.
		\]
\end{theorem}

This solves the problem of finding a formula for $\chi(X, \mathcal{N})$,
which is originally posed by Milne
\cite[Chapter III, Problem 8.10]{Mil06} when $\mathcal{N}$ is finite.
It implies our second formula,
which is a function field version of
the ``Isogeny formula (First form)''
at the end of \cite{Sch87}:

\begin{theorem}[Theorem \ref{0031}] \label{0052}
	We have
		\[
				\mu_{A_{2} / K} - \mu_{A_{1} / K}
			=
					\deg(R \Lie \mathcal{N}_{f})
				-
					c(f).
		\]
	In particular, if $f$ extends to an isogeny on the N\'eron models, then
		\[
				\mu_{A_{2} / K} - \mu_{A_{1} / K}
			=
				\deg(R \Lie \mathcal{N}_{f}).
		\]
\end{theorem}

In \cite[Theorem 1.1]{Suz25},
this theorem is applied to prove that
the Faltings height of an abelian variety over $K$ is duality invariant.

The term $\deg(R \Lie \mathcal{N}_{f})$
changes in a complicated manner
if $K$ is replaced by a finite extension.
We can replace it by a more stable invariant,
at the expense of adding and subtracting
Chai's base change conductors \cite{Cha00} of the abelian varieties $A_{1}$ and $A_{2}$.

Let $c(A_{i})$ be the base change conductor of $A_{i}$
(whose definition is recalled in
Definitions \ref{0058} and \ref{0061} \eqref{0060} below).
Let $L / K$ be a finite separable extension
over which $A_{i}$ has semistable reduction everywhere.
Let $\mathcal{N}_{f, L}$ be the object $\mathcal{N}_{f}$,
this time defined from the base change
$f \colon A_{1} \times_{K} L \to A_{2} \times_{K} L$ of $f$ to $L$.
The ratio $\deg(R \Lie \mathcal{N}_{f, L}) / [L : K]$ is
independent of the choice of $L$.
Denote it by $\deg_{X}(R \Lie \mathcal{N}_{f, K^{\sep}})$.
Our third formula is the following:

\begin{theorem}[Theorem \ref{0047}] \label{0055}
	We have
		\[
				\mu_{A_{2} / K} - \mu_{A_{1} / K}
			=
					\deg_{X}(R \Lie \mathcal{N}_{f, K^{\sep}})
				+ 
					c(A_{2}) - c(A_{1}).
		\]
\end{theorem}

Our results and methods are also useful for studying
the behavior of the Birch--Swinnerton-Dyer conjecture over function fields
under an isogeny,
including the formula for the leading coefficient of the $L$-function.
This formula (under the assumption of the finiteness of $\Sha(A)$) is proved by
Kato--Trihan \cite{KT03}.
Here we prove, without using the result of \cite{KT03},
that the validity of the leading coefficient formula is
invariant under an arbitrary isogeny on $A$,
whose degree may be divisible by the characteristic $p > 0$.
Proving this isogeny invariance is a problem posed by
Tate \cite[Section 4, First paragraph]{Tat95}
and Milne \cite[Chapter III, Problem 9.10]{Mil06}.

\begin{theorem}[Theorem \ref{0051}] \label{0057}
	The Birch--Swinnerton-Dyer formula holds for $A_{1}$
	if and only if it does for $A_{2}$.
\end{theorem}

Kato--Trihan's proof of the BSD formula
(under finiteness of Tate--Shafarevich groups)
uses heavy integral $p$-adic cohomology arguments,
while our simple proof of isogeny invariance does not.

We will also give another simpler (yet conditional) proof of Theorem \ref{0052}
without using Theorem \ref{0050} or \ref{0053}.
It instead uses the $\mu$-invariant formula given by
the third author's article with Lai--Longhi--Tan--Trihan \cite{LLSTT21}.
This formula is proved in many cases,
but not in full expected generality.
Therefore, the proof based on this $\mu$-invariant formula,
though simpler, is conditional.

Actually a certain version of the $\mu$-invariant can be geometrically defined
even if the constant field $k$ is not finite but,
more generally, perfect of positive characteristic.
It is the dimension of the ``Tate--Shafarevich scheme'' constructed
by the third author \cite{Suz20a}
based on ``the ind-rational pro-\'etale site of $k$''
$\Spec k^{\ind\rat}_{\pro\et}$.
That it agrees with the $\mu$-invariant when $k$ is finite
is proved in \cite[Theorem 2.4.3]{LLSTT21}.
In this paper, we will review some notation about $\Spec k^{\ind\rat}_{\pro\et}$
and then formulate and prove geometric versions of
Theorems \ref{0050}, \ref{0053}, \ref{0052} and \ref{0055}
for general perfect $k$.
We will see that these versions immediately imply the original versions above.

\begin{acknowledgments}
The second author thanks Aditya Karnataki for inviting him to Chennai Mathematical Institute in July 2023. He also thanks Anwesh Ray for discussing similar questions dealt in this article in the number field setup for non-commutative Lie extensions which inspired him to ask these questions in the function field case. He also gratefully acknowledges support from the Inspire research grant. We thank the reviewer for the detailed review and careful comments. The first author acknowledges the support received from the HRI postdoctoral fellowship.
\end{acknowledgments}


\section{Notation and conventions}

All groups and group schemes are assumed to be commutative.

Let $\Ab$ be the category of abelian groups.
For an abelian category $\mathcal{A}$,
its derived category is denoted by $D(\mathcal{A})$
and its bounded derived category by $D^{b}(\mathcal{A})$.
The mapping cone of a morphism $A \to B$ between complexes in $\mathcal{A}$
is denoted by $[A \to B]$
(hence, its negative shift $[A \to B][-1]$ is the mapping fiber).
If we say $A \to B \to C$ is a distinguished triangle in $D(\mathcal{A})$,
we implicitly assume that a morphism $C \to A[1]$ to the shift of $A$ is given
and the resulting triangle $A \to B \to C \to A[1]$ is distinguished.
If $F \colon \mathcal{A} \to \mathcal{B}$ is an additive functor
from an abelian category with enough K-injectives
to another abelian category,
then its right derived functor is denoted by
$R F \colon D(\mathcal{A}) \to D(\mathcal{B})$.

We denote the category of sheaves of abelian groups 
on a Grothendieck site $S$ by $\Ab(S)$.
Denote $D^{\ast}(S) = D^{\ast}(\Ab(S))$
with $\ast = $ (blank) or $b$.
For $X \in S$ and $A \in D(S)$,
denote by $R \Gamma(X, A)$ the cohomology of $S$ at $X$
with coefficients in $A$.

For a commutative ring $R$ and an $R$-module $M$,
let $\length_{R} M$ be the length of $M$.

For a quasi-compact, irreducible, regular scheme $X$ of dimension $1$
with perfect residue fields at closed points,
let $X_{0}$ be the set of closed points of $X$.
For $v \in X_{0}$, let $\Order_{v}$ be the completed local ring of $X$ at $v$,
with fraction field $K_{v}$ and residue field $k(v)$.

For a power $q = p^{e}$ of a prime $p$
and a nonzero rational number $r$,
let $v_{q}(r)$ be the $q$-valuation of $r$,
which is the $p$-adic valuation divided by $e$
(so $v_{q}(q) = 1$).


\section{The ind-rational pro-\'etale site}

Let $k$ be a perfect field of characteristic $p > 0$.
We recall some notation about perfections of algebraic groups
and the ind-rational pro-\'etale site
$\Spec k^{\ind\rat}_{\pro\et}$ from \cite{Suz20b} and \cite{Suz20a}.
See also \cite[Section 3]{OS26}.

The perfection of a $k$-algebra $R$ (respectively a $k$-scheme $X$)
is the direct limit $R \to R \to \cdots$
(respectively the inverse limit $\cdots \to X \to X$)
along the absolute Frobenius morphisms,
equipped with the natural $k$-algebra (respectively $k$-scheme) structure.

Let $\Loc \Alg / k$ be the category of perfections
of smooth group schemes over $k$
with group scheme morphisms.%
\footnote{
	As in \cite[Section 2.1, Footnote 3]{Suz22},
	the ``$\Loc$'' stands for ``locally.''
	In \cite[Section 1.2, Definition 2]{Ser60}
	(see also \cite[Section 1.4, Proposition 10]{Ser60}),
	Serre calls the perfection of an algebraic group
	a ``quasi-algebraic group.''
	By extension, a possible terminology for an object of our category $\Loc \Alg / k$
	is a ``locally quasi-algebraic group'' over $k$.
	Most group schemes over $k$ in this paper (and \cite{Suz22}) are perfect.
	Some group schemes over $k$ in \cite{Suz22} are pro-algebraic
	and not necessarily the perfections of group schemes locally of finite type.
	These leave us ``locally'' and ``algebraic'' (without ``quasi'')
	as the only significant adjectives here and in \cite{Suz22},
	whence the notation $\Loc \Alg / k$.
}
Let $\Alg / k \subset \Loc \Alg / k$ be the full (abelian) subcategory
of quasi-compact groups.
Let $\Loc \Alg_{\uc} / k \subset \Loc \Alg / k$ be the full subcategory
of groups with unipotent identity component.
Let $\Alg_{\uc} / k := \Alg / k \cap \Loc \Alg_{\uc} / k$.

A $k$-algebra $k'$ is said to be \emph{rational}
if it can be written as $k'_{1} \times \dots \times k'_{n}$,
where each $k'_{i}$ is the perfection
of a finitely generated extension field over $k$.
A $k$-algebra is said to be \emph{ind-rational}
if it can be written as a filtered direct limit of rational $k$-algebras.
Let $\Spec k^{\ind\rat}_{\pro\et}$ be
(the opposite of) the category of ind-rational $k$-algebras
with $k$-algebra homomorphisms equipped with the pro-\'etale topology.
We denote $\Ab(\Spec k^{\ind\rat}_{\pro\et}) =: \Ab(k^{\ind\rat}_{\pro\et})$
and $D^{\ast}(\Spec k^{\ind\rat}_{\pro\et}) =: D^{\ast}(k^{\ind\rat}_{\pro\et})$
with $\ast = $ (blank) or $b$.
Similarly, the cohomology functor
$R \Gamma(\Spec k', \var)$ on this site at $\Spec k'$
for any ind-rational $k$-algebra $k'$
is also denoted by $R \Gamma(k', \var)$.
The natural Yoneda functor $\Alg / k \to \Ab(k^{\ind\rat}_{\pro\et})$ is fully faithful exact
and induces a fully faithful embedding
$D^{b}(\Alg / k) \into D(k^{\ind\rat}_{\pro\et})$
(\cite[Proposition 2.3.4]{Suz20b}).
The bounded derived category of \'etale group schemes over $k$
is also a full subcategory of $D(k^{\ind\rat}_{\pro\et})$
(\cite[Proposition 2.3.4]{Suz20b}).
In particular, the natural Yoneda functor
$\Loc \Alg / k \to \Ab(k^{\ind\rat}_{\pro\et})$ is fully faithful exact.

Let $X$ be a geometrically connected, proper, smooth curve over $k$.
Let $U \subset X$ be a dense open subscheme.
Let $U_{\fppf}$ be the fppf site of $U$.
The cohomology functor $R \Gamma(U, \var)$ for $U$ in this paper
always means the cohomology of $U_{\fppf}$,
while the cohomology $R \Gamma(k, \var)$ for $k$ is
the cohomology of $\Spec k^{\ind\rat}_{\pro\et}$
as agreed above.

We recall the left exact functor
	\[
			\alg{\Gamma}(U, \var)
		\colon
			\Ab(U_{\fppf})
		\to
			\Ab(k^{\ind\rat}_{\pro\et})
	\]
defined in \cite[Section 2.7]{Suz20a}.
For $G \in \Ab(U_{\fppf})$,
define $\alg{\Gamma}(U, G) \in \Ab(k^{\ind\rat}_{\pro\et})$
to be the pro-\'etale sheafification of the presheaf
$k' \mapsto G(U \times_{k} k')$,
where $k'$ runs over ind-rational $k$-algebras.
Note that the presheaf $k' \mapsto G(U \times_{k} k')$ is an \'etale sheaf
since the fppf topology is finer than the \'etale topology.
It is not a pro-\'etale sheaf in general, however,
since the fppf topology is not finer than the pro-\'etale topology.
Thus the pro-\'etale sheafification in the definition of $\alg{\Gamma}(U, G)$
is necessary in general.
Varying $G$, this construction defines the functor $\alg{\Gamma}(U, \var)$.
Denote its $i$-th right derived functor by $\alg{H}^{i}(U, \var)$.

For any $G \in D(U_{\fppf})$, there is a canonical morphism
	\begin{equation} \label{0065}
			R \Gamma(U, G)
		\to
			R \Gamma(k, R \alg{\Gamma}(U, G))
	\end{equation}
in $D(\Ab)$ by \cite[Tag 05T2 (1)]{Sta26},
where $R \Gamma(U, \var)$ on the left-hand side is the fppf cohomology of $U$
and $R \Gamma(k, \var)$ on the left-hand side is the pro-\'etale cohomology of $k$
as agreed above.
This morphism is an isomorphism
if $G$ is a group scheme locally of finite type over $U$
(or a bounded complex thereof)
as proved in \cite[Proposition 2.7.8]{Suz20a} and the paragraph thereafter.
(The cited proposition says that under this assumption on $G$, the functor
$k' \mapsto H^{i}(U \times_{k} k', G)$
for any $i$ commutes with filtered direct limits
and hence its \'etale sheafification is already a pro-\'etale sheaf,
which is $\alg{H}^{i}(U, G)$.
This allows us to apply the general result
\cite[Tag 015M]{Sta26}
on composition of derived functors
to obtain \eqref{0065}.)

Let $v \in X_{0}$.
Let $\Spec \Order_{v, \fppf}$ be the fppf site of $\Order_{v}$.
Let $\Ab(\Order_{v, \fppf})$ be the category of sheaves of abelian groups
on $\Spec \Order_{v, \fppf}$.
We recall the left exact functor
	\[
			\alg{\Gamma}(\Order_{v}, \var)
		\colon
			\Ab(\Order_{v, \fppf})
		\to
			\Ab(k(v)^{\ind\rat}_{\pro\et})
	\]
defined in \cite[Section 2.5]{Suz20a}.%
\footnote{
	The notation for this functor in \cite[Section 2.5]{Suz20a} is actually
	$\alg{\Gamma}(\Hat{\Order}_{v}, \var)$
	(with a hat)
	and there is a ``henselian'' version
	$\alg{\Gamma}(\Order^{h}_{v}, \var)$.
	We do not explicitly need the henselian version in this paper.
	Therefore, we use the notation $\alg{\Gamma}(\Order_{v}, \var)$ for simplicity.
}
The Teichm\"uller section defines
a canonical $k(v)$-algebra structure on $\Order_{v}$.
For an ind-rational $k(v)$-algebra $k'$,
define
	\[
			\Order_{v} \ctensor_{k(v)} k'
		:=
			\invlim_{n}
				\bigl(
						(\Order_{v} / \ideal{p}_{v}^{n})
					\tensor_{k(v)}
						k'
				\bigr),
	\]
where $\ideal{p}_{v}$ is the maximal ideal of $\Order_{v}$.
For $G \in \Ab(\Order_{v, \fppf})$,
define $\alg{\Gamma}(\Order_{v}, G) \in \Ab(k(v)^{\ind\rat}_{\pro\et})$
to be the pro-\'etale sheafification of the presheaf
$k' \mapsto G(\Order_{v} \ctensor_{k(v)} k')$,
where $k'$ runs over ind-rational $k(v)$-algebras.
Varying $G$, this construction defines the functor $\alg{\Gamma}(\Order_{v}, \var)$.
Denote its $i$-th right derived functor by $\alg{H}^{i}(\Order_{v}, \var)$.

For any $G \in D(\Order_{v, \fppf})$, there is a canonical morphism
	\[
			R \Gamma(\Order_{v}, G)
		\to
			R \Gamma(k(v), R \alg{\Gamma}(\Order_{v}, G))
	\]
in $D(\Ab)$ by \cite[Tag 05T2 (1)]{Sta26}.
It is an isomorphism
if $G$ is a smooth group scheme or a finite flat group scheme
as shown in the paragraph before \cite[Proposition 2.5.2]{Suz20a}.
It is also an isomorphism
if $G$ is a separated, quasi-finite, flat group scheme
since such a group is an extension of an \'etale (hence smooth) group scheme
by a finite flat group scheme
(the finite part of $G$;
see the paragraph before \cite[Section 7.3, Proposition 3]{BLR90}).

Let $f \colon \Spec k(v) \to \Spec k$ be the natural morphism.
Let
	\[
			f
		\colon
			\Spec k(v)^{\ind\rat}_{\pro\et}
		\to
			\Spec k^{\ind\rat}_{\pro\et}
	\]
be the induced morphism of sites
(defined by the functor $k' \tensor_{k} k(v) \mapsfrom k'$
on the underlying categories).
Set
	\[
				\alg{\Gamma}(\Order_{v} / k, \var)
			:=
				f_{\ast}
				\alg{\Gamma}(\Order_{v}, \var)
		\colon
			\Ab(\Order_{v, \fppf})
		\to
			\Ab(k^{\ind\rat}_{\pro\et})
	\]
as in the end of \cite[Section 2.5]{Suz20a}.
Denote its $i$-th right derived functor by
$\alg{H}^{i}(\Order_{v} / k, \var)$.

Let $\Ab(K_{v, \fppf})$ be the category of sheaves on the fppf site of $K_{v}$.
Using $K_{v}$ in place of $\Order_{v}$
and $(\Order_{v} \ctensor_{k(v)} k') \tensor_{\Order_{v}} K_{v}$
in place of $\Order_{v} \ctensor_{k(v)} k'$,
we can similarly define left exact functors
	\[
			\alg{\Gamma}(K_{v}, \var)
		\colon
			\Ab(K_{v, \fppf})
		\to
			\Ab(k(v)^{\ind\rat}_{\pro\et}),
	\]
	\[
			\alg{\Gamma}(K_{v} / k, \var)
		\colon
			\Ab(K_{v, \fppf})
		\to
			\Ab(k^{\ind\rat}_{\pro\et}).
	\]

In the paragraph before \cite[Proposition 2.7.4]{Suz20a},
a certain full triangulated subcategory
$D(U_{\fppf})_{\ca} \subset D(U_{\fppf})$
(of objects satisfying ``cohomological approximation'') is defined.
Its definition is recalled in Remark \ref{0080} below,
but the details of the definition do not matter in this paper.
Here are what we need to know about $D(U_{\fppf})_{\ca}$:
By definition, it is \'etale-local,
namely, if $\{U_{i} \to U\}_{i}$ is an \'etale covering family of $U$
(with $U_{i}$ irreducible and separated)
and the restriction of $G$ to $U_{i, \fppf}$ belongs to
$D(U_{i, \fppf})_{\ca}$ for any $i$,
then $G \in D(U_{\fppf})_{\ca}$.
Furthermore, any smooth group scheme or finite flat group scheme over $U$
(or a bounded complex thereof)
belongs to $D(U_{\fppf})_{\ca}$
by \cite[Proposition 2.7.5]{Suz20a}.
This also applies to $G$ being a separated, quasi-finite, flat group scheme over $U$
since such a $G$ is \'etale locally
an extension of an \'etale group scheme
by a finite flat group scheme.

For $G \in D(U_{\fppf})$, let $R \alg{\Gamma}_{c}(U, G)$ be
the canonical mapping fiber of the natural restriction morphism
	\[
			R \alg{\Gamma}(U, G)
		\to
			\bigoplus_{v \in X \setminus U}
				R \alg{\Gamma}(K_{v} / k, G)
	\]
in $D(k^{\ind\rat}_{\pro\et})$.
See the paragraph after \cite[Proposition 2.7.2]{Suz20a}
for the details about how the mapping fiber is chosen canonically.
This defines a triangulated functor
	\[
			R \alg{\Gamma}_{c}(U, \var)
		\colon
			D(U_{\fppf})
		\to
			D(k^{\ind\rat}_{\pro\et})
	\]
and fits in a canonical distinguished triangle
	\[
			R \alg{\Gamma}_{c}(U, G)
		\to
			R \alg{\Gamma}(U, G)
		\to
			\bigoplus_{v \in X \setminus U}
				R \alg{\Gamma}(K_{v} / k, G)
	\]
in $D(k^{\ind\rat}_{\pro\et})$ functorial in $G$.
Denote $\alg{H}^{i}_{c}(U, \var) = H^{i} R \alg{\Gamma}_{c}(U, \var)$.

If $G \in D(X_{\fppf})_{\ca}$,
then there exists a canonical distinguished triangle
	\begin{equation} \label{0039}
			R \alg{\Gamma}_{c}(U, G)
		\to
			R \alg{\Gamma}(X, G)
		\to
			\bigoplus_{v \in X \setminus U}
				R \alg{\Gamma}(\Order_{v} / k, G)
	\end{equation}
in $D(k^{\ind\rat}_{\pro\et})$ functorial in $G$
by \cite[Proposition 2.7.4]{Suz20a}.
Applying $R \Gamma(k, \var)$ recovers
the compact support fppf cohomology distinguished triangle
	\[
			R \Gamma_{c}(U, G)
		\to
			R \Gamma(X, G)
		\to
			\bigoplus_{v \in X \setminus U}
				R \Gamma(\Order_{v}, G)
	\]
in $D(\Ab)$ given in
\cite[Chapter III, Proposition 0.4 (c), Remark 0.6 (b)]{Mil06}
(see also \cite[Proposition 2.1 (3)]{DH19}).

We will need the following result
to pass from the general base field case
to the finite base field case.

\begin{proposition} \label{0005}
	Assume that $k$ is finite with $q$ elements.
	\begin{enumerate}
		\item \label{0001}
			Let $G \in D^{b}(\Alg / k)$.
			Then $H^{i}(k, G)$ is finite for all $i$
			and zero for almost all $i$.
		\item \label{0002}
			Let $G \in D^{b}(\Alg_{\uc} / k)$.
			Then
				\[
						\prod_{i}
							(\# H^{i}(k, G))^{(-1)^{i}}
					=
						q^{
							\sum_{i}
								(-1)^{i}
								\dim H^{i}(G)
						}
				\]
			(where $H^{i}(G)$ is the $i$-th cohomology of $G$ as a complex).
	\end{enumerate}
\end{proposition}

\begin{proof}
	First, the cohomology here is a priori pro-\'etale cohomology
	$H^{i}(k_{\pro\et}, G)$ as agreed above.
	But for $G \in D^{b}(\Alg / k)$,
	pro-\'etale cohomology agrees with \'etale cohomology:
		\[
				H^{i}(k_{\pro\et}, G)
			\cong
				H^{i}(k_{\et}, G)
		\]
	by \cite[Proposition (2.1.2) (g)]{Suz20b}.
	If, moreover, $G$ is the perfection of a smooth group scheme $G_{0}$ over $k$,
	then $G(k') \cong G_{0}(k')$ for any \'etale $k$-algebra $k'$,
	so $H^{i}(k_{\et}, G) \cong H^{i}(k_{\et}, G_{0})$.
	These will allow us below to freely apply some classical results
	on Galois cohomology of smooth group schemes
	to understand $H^{i}(k, G)$.
	
	\eqref{0001}
	We may assume that $G \in \Alg / k$.
	Then $H^{0}(k, G) = G(k)$ and $H^{1}(k, G) = H^{1}(k, \pi_{0}(G))$ are finite
	and $H^{i}(k, G) = 0$ for $i \ge 2$ by Lang's theorem.
	
	\eqref{0002}
	Both sides of the desired equality are additive in distinguished triangles.
	Hence, we may assume that $G \in \Alg_{\uc} / k$.
	Let $F_{k} \colon G \to G$ be
	the $q$-th power Frobenius for $G$ over $k$.
	If $G$ is finite \'etale, then the exact sequence
		\[
				0
			\to
				H^{0}(k, G)
			\to
				G(\closure{k})
			\stackrel{F_{k} - 1}{\to}
				G(\closure{k})
			\to
				H^{1}(k, G)
			\to
				0
		\]
	shows that $H^{0}(k, G)$ and $H^{1}(k, G)$ have the same cardinality.
	With $H^{i}(k, G) = 0$ for $i \ne 0, 1$, we get the result in this case.
	
	If $G$ is connected (and hence unipotent),
	we may assume that $G = \Ga$.
	Then $\# H^{0}(k, G) = \# k = q$
	and $H^{i}(k, G) = 0$ for $i \ne 0$,
	which give the result.
\end{proof}

\begin{remark} \label{0080}
	Let $U \subset X$ be a dense open subscheme.
	We recall the definition of the full subcategory
	$D(U_{\fppf})_{\ca} \subset D(U_{\fppf})$
	from the paragraph before \cite[Proposition 2.7.4]{Suz20a}.
	It consists of objects $G \in D(U_{\fppf})$
	satisfying the property that the natural morphism
		\[
				R \Gamma_{v}(\Order_{v}^{h} \tensor^{h}_{k(v)} k', G)
			\to
				R \Gamma_{v}(\Order_{v} \ctensor_{k(v)} k', G)
		\]
	in $D(\Ab)$ is an isomorphism for any $v \in U_{0}$
	and any ind-rational w-contractible%
	\footnote{
		W-contractible means that
		any faithfully flat ind-\'etale ring homomorphism
		$k' \into k''$ admits a retract $k'' \onto k'$.
	}
	$k(v)$-algebra $k'$,
	where $\Order_{v}^{h}$ is the henselian local ring of $U$ at $v$
	with maximal ideal $\ideal{p}^{h}_{v}$ and fraction field $K_{v}^{h}$;
	$\Order_{v}^{h} \tensor^{h}_{k(v)} k'$ is the henselization of
	$\Order_{v}^{h} \tensor_{k(v)} k'$ with respect to the ideal
	$\ideal{p}^{h}_{v} \tensor_{k(v)} k'$;
	$R \Gamma_{v}(\Order_{v}^{h} \tensor^{h}_{k(v)} k', \var)$ denotes
	the fppf cohomology functor for $\Order_{v}^{h} \tensor^{h}_{k(v)} k'$
	with support corresponding to the ideal $\ideal{p}^{h}_{v} \tensor_{k(v)} k'$,
	namely, the canonical mapping fiber of the morphism
		\[
				R \Gamma(\Order_{v}^{h} \tensor^{h}_{k(v)} k', \var)
			\to
				R \Gamma \bigl(
						(\Order_{v}^{h} \tensor^{h}_{k(v)} k')
					\tensor_{\Order_{v}^{h}}
						K_{v}^{h},
					\var
				\bigr);
		\]
	and $R \Gamma_{v}(\Order_{v} \ctensor_{k(v)} k', \var)$ denotes
	a similarly defined functor for $\Order_{v} \ctensor_{k(v)} k'$.
	This property depends only on the restrictions of $G$
	to the fppf sites of the strict henselizations of
    $\Order_{v}^{h}$ for all $v \in U_{0}$.
\end{remark}


\section{Conductors of isogenies in the local case}
\label{0017}

In this section,
we will define a number that we call the ``conductor,'' $c(f)$,
for an isogeny $f$ between abelian varieties over local fields.
Such an isogeny does not extend to an isogeny on the N\'eron models
(in the sense of having finite kernel on the special fibers) in general.
Our conductor measures this defect.
The kernel of the morphism induced on the N\'eron models
is not finite, not even quasi-finite, in general.
We first extract a certain quasi-finite group scheme $\mathcal{N}_{f}$
from this non-quasi-finite group scheme.
With this, we define $c(f)$ and relate it to
the Liu--Lorenzini--Raynaud constructions \cite[Theorem 2.1 (b)]{LLR04}.
Both $\mathcal{N}_{f}$ and $c(f)$ contribute to the change in $\mu$-invariants
in the global situation.

Let $K$ be a complete discrete valuation field
with ring of integers $\Order_{K}$ and perfect residue field $k$ of positive characteristic.
Let $f \colon A_{1} \to A_{2}$ be an isogeny of abelian varieties over $K$
with kernel $N_{f}$.
We have an exact sequence
	\[
			0
		\to
			N_{f}
		\to
			A_{1}
		\to
			A_{2}
		\to
			0.
	\]
For $i = 1, 2$, let $\mathcal{A}_{i}$ be the N\'eron model over $\Order_{K}$ of $A_{i}$.
Then $f$ induces a morphism
$\Tilde{f} \colon \mathcal{A}_{1} \to \mathcal{A}_{2}$ of group schemes over $\Order_{K}$.
Let $\mathcal{N}_{f} \subset \mathcal{A}_{1}$ be
the schematic closure of $N_{f}$ in $\mathcal{A}_{1}$,
which is a separated, quasi-finite, flat closed group subscheme over $\Order_{K}$.
The quotient $\mathcal{A}_{2}' := \mathcal{A}_{1} / \mathcal{N}_{f}$
as an fppf sheaf over $\Order_{K}$
is representable by a smooth, separated group scheme of finite type
by \cite[Theorem 4.C]{Ana73}.
Since $\Tilde{f} \colon \mathcal{A}_{1} \to \mathcal{A}_{2}$ annihilates $\mathcal{N}_{f}$,
we have an induced morphism
$\mathcal{A}_{2}' \to \mathcal{A}_{2}$:
	\[
		\begin{CD}
				0
			@>>>
				\mathcal{N}_{f}
			@>>>
				\mathcal{A}_{1}
			@>>>
				\mathcal{A}_{2}'
			@>>>
				0
			\\ @. @. @. @VVV \\
			@. @. @. \mathcal{A}_{2}.
		\end{CD}
	\]
The morphism $\mathcal{A}_{2}' \to \mathcal{A}_{2}$ is
an isomorphism on the generic fibers.

\begin{definition}
	Define $c(f)$ to be the length of
	$(\Lie \mathcal{A}_{2}) / (\Lie \mathcal{A}_{2}')$
	as an $\Order_{K}$-module
	and call it the \emph{conductor} of the isogeny $f$.
\end{definition}

By \cite[Theorem 2.1 (b)]{LLR04}
(see also \cite[Proposition 2.4 and Remark 2.5]{OS26}),
there exists a canonical smooth group scheme $D = D_{f}$ of finite type over $k$
such that $D_{f}(\closure{k})$ is canonically isomorphic to the cokernel of
	$
			\mathcal{A}_{2}'(\Hat{\Order}_{K}^{\ur})
		\into
			\mathcal{A}_{2}(\Hat{\Order}_{K}^{\ur})
		=
			A_{2}(\Hat{K}^{\ur})
	$
as $\Gal(\closure{k} / k)$-modules
and we have $c(f) = \dim D_{f}$,
where $\Hat{\Order}_{K}^{\ur}$ is the completed maximal unramified extension of $\Order_{K}$
and $\Hat{K}^{\ur}$ its fraction field.
More explicitly,
$D_{f}$ is given by the cokernel of the morphism
$\Gr_{\Order_{K}} \mathcal{A}_{2}' \to \Gr_{\Order_{K}} \mathcal{A}_{2}$
induced on the infinite-level Greenberg transforms of $\mathcal{A}_{2}'$ and $\mathcal{A}_{2}$
(\cite[Section 14]{BGA18}, \cite[Remark 2.5]{OS26}):
	\begin{equation} \label{0011}
			\Gr_{\Order_{K}} \mathcal{A}_{2}'
		\to
			\Gr_{\Order_{K}} \mathcal{A}_{2}
		\to
			D_{f}
		\to
			0.
	\end{equation}
The kernel of the first morphism
$\Gr_{\Order_{K}} \mathcal{A}_{2}' \to \Gr_{\Order_{K}} \mathcal{A}_{2}$
is pro-infinitesimal (since it is injective on geometric points).
This kernel is zero if $K$ has equal characteristic
by the explicit description of Greenberg transforms
as Weil restrictions in this case.
In general, $D_{f}$ has unipotent identity component
since there is an isogeny $A_{2} \to A_{1}$
that composes with $f$ to give multiplication by $\deg(f)$
(\cite[Section 7.3, Lemma 5]{BLR90}) and hence
$\mathcal{A}_{1} \to \mathcal{A}_{2}' \to \mathcal{A}_{2}$
are isomorphisms up to bounded torsion.

Recall from \cite[Section 7.3, Definition 4]{BLR90} that
$\Tilde{f} \colon \mathcal{A}_{1} \to \mathcal{A}_{2}$ is called an isogeny
if the induced morphism $\mathcal{A}_{1, k} \to \mathcal{A}_{2, k}$ on the special fibers
is an isogeny (namely, has finite kernel).
The number $c(f)$ measures the defect of $\Tilde{f}$ being an isogeny:

\begin{proposition} \label{0000}
	We have $c(f) = 0$ if and only if
	$\Tilde{f} \colon \mathcal{A}_{1} \to \mathcal{A}_{2}$ is an isogeny,
	in which case its kernel is $\mathcal{N}_{f}$
	and the morphism $\mathcal{A}_{2}' \to \mathcal{A}_{2}$ is an open immersion.
\end{proposition}

\begin{proof}
	This is a special case of
	\cite[Lemmas B.4 and B.7]{Ces16}
	(which are stated for N\'eron models
	but whose proofs work for non-N\'eron models
	such as $\mathcal{A}_{2}'$).
\end{proof}

Note that $\Tilde{f} \colon \mathcal{A}_{1} \to \mathcal{A}_{2}$ is an isogeny
if $A_{1}$ has semistable reduction or
the degree of $f \colon A_{1} \to A_{2}$ is prime to the characteristic of $k$
(\cite[Section 7.3, Proposition 6]{BLR90}).
The converse holds if $f$ is multiplication by a natural number on $A_{1} = A_{2}$
(\cite[Chapter III, Corollary C.9]{Mil06}).
This gives examples with $c(f) \ne 0$.

We also recall Chai's base change conductors.
Let $A$ be an abelian variety over $K$.
Let $L$ be a finite (not necessarily separable) extension of $K$
with ring of integers $\Order_{L}$ and ramification index $e_{L / K}$
over which $A$ has semistable reduction.
Let $\mathcal{A}$ be the N\'eron model over $\Order_{K}$ of $A$
and $\mathcal{A}_{L}$ the N\'eron model
over $\Order_{L}$ of $A \times_{K} L$.

\begin{definition}[{\cite[Section 2.4]{Cha00}}] \label{0058}
	Define a rational number $c(A)$ by
		\[
				c(A)
			=
				\frac{1}{e_{L / K}}
				\length_{\Order_{L}}
					\frac{
						\Lie \mathcal{A}_{L}
					}{
						\Lie \mathcal{A} \tensor_{\Order_{K}} \Order_{L}
					},
		\]
	and call it the \emph{base change conductor} of $A$.
\end{definition}

This is independent of the choice of $L$
as mentioned in \cite[Section 2.4]{Cha00}
(use \cite[Section 7.4, Corollary 4]{BLR90}).
We only need a separable $L / K$ in this paper.


\section{Conductors of isogenies in the global case}
\label{0018}

We will globalize the constructions of the previous section.

Let $X$ be a quasi-compact, irreducible, regular scheme of dimension $1$
with perfect residue fields of positive characteristic at closed points.
Let $K$ be its function field.
Let $f \colon A_{1} \to A_{2}$ be an isogeny of abelian varieties over $K$.
For any $v \in X_{0}$,
it gives an isogeny $f \colon A_{1} \times_{K} K_{v} \to A_{2} \times_{K} K_{v}$
over the local field $K_{v}$ at $v$.
Applying the constructions in the previous section,
we have the corresponding group scheme $D_{f, v}$
and the conductor $c_{v}(f) = \dim D_{f, v}$.
Since $A_{1}$ has semistable reduction almost everywhere,
we know that $c_{v}(f) = 0$ for almost all $v$
by Proposition \ref{0000} and the paragraph thereafter.

\begin{definition} \label{0048} \mbox{}
	\begin{enumerate}
	\item
		Define a divisor on $X$ by $\mathfrak{c}(f) = \sum_{v \in X_{0}} c_{v}(f) v$
		and call it the \emph{conductor divisor} of $f$.
	\item
		When $X$ is a geometrically connected, proper, smooth curve over a perfect field $k$,
		define $c(f)$ to be the degree of $\mathfrak{c}(f)$ (with respect to $k$)
		and call it the \emph{conductor} of $f$.
	\end{enumerate}
\end{definition}

Here the degree of a single closed point $v \in X_{0}$ for example
in the second definition is not $1$ but $[k(v) : k]$
and hence it depends on $k$.

Let $N_{f}$ be the kernel of $f$.
Let $\mathcal{A}_{i}$ be the N\'eron model over $X$ of $A_{i}$.
Let $\Tilde{f} \colon \mathcal{A}_{1} \to \mathcal{A}_{2}$ be
the induced morphism.
Let $\mathcal{N}_{f} \subset \mathcal{A}_{1}$ be the schematic closure of $N_{f}$,
which is a separated, quasi-finite, flat closed group subscheme over $X$.
The quotient $\mathcal{A}_{2}' := \mathcal{A}_{1} / \mathcal{N}_{f}$
as an fppf sheaf over $X$
is representable by a smooth, separated group scheme of finite type
by \cite[Theorem 4.C]{Ana73}.
We have an induced morphism $\mathcal{A}_{2}' \to \mathcal{A}_{2}$:
	\begin{equation} \label{0008}
		\begin{CD}
				0
			@>>>
				\mathcal{N}_{f}
			@>>>
				\mathcal{A}_{1}
			@>>>
				\mathcal{A}_{2}'
			@>>>
				0
			\\ @. @. @. @VVV \\
			@. @. @. \mathcal{A}_{2}.
		\end{CD}
	\end{equation}

\begin{proposition} \label{0010}
	We have $\mathfrak{c}(f) = 0$ if and only if
	$\Tilde{f} \colon \mathcal{A}_{1} \to \mathcal{A}_{2}$ is an isogeny,
	in which case its kernel is $\mathcal{N}_{f}$
	and the morphism $\mathcal{A}_{2}' \to \mathcal{A}_{2}$ is an open immersion.
\end{proposition}

\begin{proof}
	This follows from the corresponding local statement,
	Proposition \ref{0000}.
\end{proof}

We also globalize base change conductors.
Let $A$ be an abelian variety over $K$.
For each $v \in X_{0}$,
we have the base change conductor $c_{v}(A)$ of $A \times_{K} K_{v}$.

\begin{definition} \label{0061} \mbox{}
	\begin{enumerate}
	\item \label{0059}
		Define a $\Q$-divisor on $X$ by
		$\mathfrak{c}(A) = \sum_{v \in X_{0}} c_{v}(A) v$
		and call it the \emph{base change conductor divisor} of $A$.
	\item \label{0060}
		When $X$ is a geometrically connected, proper, smooth curve over a perfect field $k$,
		define $c(A)$ to be the degree of $\mathfrak{c}(A)$ (with respect to $k$)
		and call it the \emph{base change conductor} of $A$.
	\end{enumerate}
\end{definition}

\begin{remark}
	The basic idea in the constructions in this and the previous sections is
	to imagine to have a kind of ``exact sequence''
		\[
				\text{`` }
				0
			\to
				\mathcal{N}_{f}
			\to
				\mathcal{A}_{1}
			\to
				\mathcal{A}_{2}
			\to
				\bigoplus_{v \in X_{0}}
					i_{v, \ast} D_{f, v}
			\to
				0
				\text{ ''},
		\]
	where $i_{v} \colon \Spec k(v) \into X$ is the inclusion.
	It is not really clear in what category this supposed exact sequence actually lives,
	since $D_{f, v}$ is not an \'etale scheme
	and $\mathcal{A}_{2}' \to \mathcal{A}_{2}$ is not injective
	on the special fibers in general.
	This sequence is trying to decompose
	the information about $\mathcal{A}_{1} \to \mathcal{A}_{2}$
	into the information about $\mathcal{N}_{f}$ and $D_{f, v}$.
	This idea works out in practice,
	as we will see in the rest of the paper.
\end{remark}


\section{%
	\texorpdfstring{Change in $\mu$ by flat Euler characteristics: general base field case}
	{Change in mu by flat Euler characteristics: general base field case}
}
\label{0037}

Let $X$ be a geometrically connected, proper, smooth curve
over a perfect field $k$ of characteristic $p > 0$.
Let $K$ be its function field.
In this and the next sections,
we will formulate and prove a formula for the change in $\mu$-invariants under isogeny
in terms of certain Euler characteristics in the fppf topology.
This section treats the version of the statement
for a general, not necessarily finite base field $k$.

For an fppf sheaf $\mathcal{F}$ on $X$,
if $\alg{H}^{i}(X, \mathcal{F}) \in \Loc \Alg / k$ for all $i$
and $\alg{H}^{i}(X, \mathcal{F}) = 0$ for almost all $i$,
then we define
	\[
			\chi^{0}(X, \mathcal{F})
		:=
			\sum_{i}
				(-1)^{i}
				\dim \alg{H}^{i}(X, \mathcal{F}).
	\]
(The notation $0$ indicates that only the identity components matter.)
If $\mathcal{F}'$ and $\mathcal{F}''$ are other such sheaves
and $0 \to \mathcal{F}' \to \mathcal{F} \to \mathcal{F}'' \to 0$
is an exact sequence,
then%
\footnote{
	The citation ``[GRS24, (6.1)]'' before \cite[Proposition 2.4]{Suz25}
	refers to an older, preprint version of this paper
	and corresponds to \eqref{0036} below.
}
	\begin{equation} \label{0067}
			\chi^{0}(X, \mathcal{F})
		=
			\chi^{0}(X, \mathcal{F}') + \chi^{0}(X, \mathcal{F}'').
	\end{equation}
For an abelian variety $A$ over $K$ with N\'eron model $\mathcal{A}$ over $X$,
we have $\alg{H}^{i}(X, \mathcal{A}) \in \Loc \Alg / k$ for all $i$
and $\alg{H}^{i}(X, \mathcal{A}) = 0$ for $i \ne 0, 1, 2$
by \cite[Theorem 3.4.1 (1)]{Suz20a}.
Define
	\begin{equation} \label{0036}
			\mu_{A / K}
		:=
			\dim \alg{H}^{1}(X, \mathcal{A})
	\end{equation}
as in the second paragraph of \cite[Section 6.2]{LLSTT21}.
Note that the classical definition of the $\mu$-invariant does not make sense
for a general perfect (not necessarily finite) $k$.
We will discuss the finite base field case in the next section.

\begin{proposition} \label{0064}
	Let $U \subset X$ be a dense open subscheme.
	Let $\mathcal{N}$ be a quasi-compact, \'etale group scheme over $U$.
	Let $\pi \colon U_{\et} \to \Spec k_{\et}$ be the structure morphism.
	Then $R \alg{\Gamma}_{c}(U, \mathcal{N})$ is
	a bounded complex of finite \'etale group schemes over $k$
	canonically isomorphic to $R \pi_{!} \mathcal{N}$.
\end{proposition}

\begin{proof}
	Let $\Tilde{\mathcal{N}}$ be the extension-by-zero of $\mathcal{N}$ over $X$.
	We have a distinguished triangle
		\[
				R \alg{\Gamma}_{c}(U, \mathcal{N})
			\to
				R \alg{\Gamma}(X, \Tilde{\mathcal{N}})
			\to
				\bigoplus_{v \in X \setminus U}
					R \alg{\Gamma}(\Order_{v} / k, \Tilde{\mathcal{N}})
		\]
	in $D(k^{\ind\rat}_{\pro\et})$ by \eqref{0039}.
	For any $v \in X \setminus U$,
	we have $R \alg{\Gamma}(\Order_{v}, \Tilde{\mathcal{N}}) = 0$
	by \cite[Proposition (5.2.3.4)]{Suz20b},
	hence $R \alg{\Gamma}(\Order_{v} / k, \Tilde{\mathcal{N}}) = 0$.
	Therefore, we have an isomorphism
		\[
				R \alg{\Gamma}_{c}(U, \mathcal{N})
			\isomto
				R \alg{\Gamma}(X, \Tilde{\mathcal{N}}).
		\]
	On the other hand, we have
	$R \pi_{!} \mathcal{N} = R \Tilde{\pi}_{\ast} \Tilde{\mathcal{N}}$
	by definition,
	where $\Tilde{\pi} \colon X_{\et} \to \Spec k_{\et}$ is
	the structure morphism for $X$.
	Replacing $\mathcal{N}$ by $\Tilde{\mathcal{N}}$,
	we may assume that $U = X$.
	
	Let $\Spec k_{\Et}$ and $X_{\Et}$ be
	the big \'etale sites of $k$ and $X$, respectively.
	Let
		\[
				\alg{\Gamma}(U_{\et}, \var)
			\colon
				\Ab(X_{\Et})
			\to
				\Ab(k^{\ind\rat}_{\pro\et})
		\]
	be the functor sending $G \in \Ab(U_{\Et})$
	to the pro-\'etale sheafification of the presheaf
	$k' \mapsto G(X \times_{k} k')$.
	We have a natural morphism
		\begin{equation} \label{0063}
				R \alg{\Gamma}(X_{\et}, \mathcal{N})
			\to
				R \alg{\Gamma}(X, \mathcal{N})
		\end{equation}
	in $D(k^{\ind\rat}_{\pro\et})$.
	For any integer $i$, the $i$-th cohomology object of the left-hand (resp.\ right-hand) side
	is the pro-\'etale sheafification of the presheaf
	that sends an ind-rational $k$-algebra $k'$ to
	the \'etale (resp.\ fppf) cohomology of $X \times_{k} k'$
	with coefficients in $\mathcal{N}$.
	Since the fppf cohomology with coefficients in a smooth group scheme
	agrees with the \'etale cohomology
	by \cite[Theorem (11.7) (1)]{Gro68},
	we know that \eqref{0063} is an isomorphism.
	Let $\pi_{\Et} \colon X_{\Et} \to \Spec k_{\Et}$ be
	the natural morphism of sites
	(defined by the functor $X \times_{k} Z \mapsfrom Z$).
	Then $\alg{\Gamma}(U_{\et}, \var)$ is the pro-\'etale sheafification
	of the restriction of $\pi_{\Et, \ast}$ to the category of ind-rational $k$-algebras.
	Hence, it is enough to show that
	$R \pi_{\Et, \ast} \mathcal{N}$ is
	a bounded complex of finite \'etale group schemes over $k$
	canonically isomorphic to $R \pi_{\ast} \mathcal{N}$.
	But this is a consequence of the proper base change theorem
	(\cite[Tags 0DDJ (2) and 0GL0]{Sta26}).
\end{proof}

\begin{proposition} \label{0038}
	For any separated, quasi-finite, flat group scheme $\mathcal{N}$ over $X$,
	the object $\alg{H}^{i}(X, \mathcal{N})$ is in $\Alg_{\uc} / k$
	for all $i$ and zero for almost all $i$.
\end{proposition}

\begin{proof}
	Assume first that $\mathcal{N}$ is finite over $X$.
	Let $\Pro \Alg_{\uc} / k$ and $\Ind \Alg_{\uc} / k$ be
	the pro-category and the ind-category, respectively,
	of $\Alg_{\uc} / k$.
	Let $\Ind \Pro \Alg_{\uc} / k$ be the ind-category of $\Pro \Alg_{\uc} / k$.
	The Yoneda functor
	$\Ind \Pro \Alg_{\uc} / k \to \Ab(k^{\ind\rat}_{\pro\et})$
	is fully faithful by \cite[Proposition (2.3.4)]{Suz20b}.
	By \cite[Theorem 3.1.3]{Suz20a},
	we know that $\alg{H}^{i}(X, \mathcal{N})$ is in
	both $\Ind \Alg_{\uc} / k$ and $\Pro \Alg_{\uc} / k$ for all $i$
	and zero for almost all $i$.
	The intersection of these categories in $\Ind \Pro \Alg_{\uc} / k$
	is $\Alg_{\uc} / k$.
	Hence, $\alg{H}^{i}(X, \mathcal{N})$ is in $\Alg_{\uc} / k$ in this case.
	
	In general, let $N$ be the generic fiber of $\mathcal{N}$.
	Let $N^{0} \subset N$ be the identity component of $N$.
	Let $\mathcal{N}^{0} \subset \mathcal{N}$ be
	the schematic closure of $N^{0}$ in $\mathcal{N}$.
	For any $v \in X_{0}$,
	the finite part of $\mathcal{N}^{0} \times_{X} \Spec \Order_{v}$
	is an open group subscheme over $\Order_{v}$.
	Since $\mathcal{N}^{0} \times_{X} \Spec \Order_{v}$ has infinitesimal fibers,
	this open group subscheme has to be
	the whole $\mathcal{N}^{0} \times_{X} \Spec \Order_{v}$.
	Hence, $\mathcal{N}^{0}$ is finite over $X$.
	Therefore, the finite case implies that we may assume $N$ to be \'etale.
	Let $U \subset X$ be a dense open subscheme where $\mathcal{N}$ is \'etale.
	We have a distinguished triangle
		\[
				R \alg{\Gamma}_{c}(U, \mathcal{N})
			\to
				R \alg{\Gamma}(X, \mathcal{N})
			\to
				\bigoplus_{v \in X \setminus U}
					R \alg{\Gamma}(\Order_{v} / k, \mathcal{N})
		\]
	in $D(k^{\ind\rat}_{\pro\et})$ by \eqref{0039}.
	By Proposition \ref{0064} applied to
	$R \alg{\Gamma}_{c}(U, \mathcal{N})$,
	we only need to look at the third term.
	For each $v \in X \setminus U$,
	let $\mathcal{M} \subset \mathcal{N}$ be
	the finite part of $\mathcal{N}$ over $\Order_{v}$.
	Then we have
		\[
				R \alg{\Gamma}(\Order_{v} / k, \mathcal{M})
			\isomto
				R \alg{\Gamma}(\Order_{v} / k, \mathcal{N})
		\]
	by \cite[Proposition 5.2.3.5]{Suz20b}.
	Since the generic fiber of $\mathcal{M}$ is \'etale,
	we have $R \alg{\Gamma}(\Order_{v} / k, \mathcal{M}) \in D^{b}(\Alg / k)$
	by \cite[Proposition 3.7]{OS26}.
	Since $\mathcal{M}$ is killed by multiplication by some positive integer,
	we further have
	$R \alg{\Gamma}(\Order_{v} / k, \mathcal{M}) \in D^{b}(\Alg_{\uc} / k)$,
	yielding the desired result.
\end{proof}

Here is the promised formula in this section:

\begin{theorem} \label{0003} \mbox{}
	Let $f \colon A_{1} \to A_{2}$ be an isogeny of abelian varieties over $K$.
	For each $i = 1, 2$,
	let $\mathcal{A}_{i}$ be the N\'eron model of $A_{i}$ over $X$.
	Let $\mathcal{N}_{f}$ be the schematic closure
	of the kernel of $f$ in $\mathcal{A}_{1}$.
	Then we have
		\begin{equation} \label{0009}
				\mu_{A_{2} / K} - \mu_{A_{1} / K}
			=
					\chi^{0}(X, \mathcal{A}_{1})
				-
					\chi^{0}(X, \mathcal{A}_{2})
			=
					\chi^{0}(X, \mathcal{N}_{f})
				-
					c(f).
		\end{equation}
	In particular, if $f$ extends to an isogeny on the N\'eron models, then
		\[
				\mu_{A_{2} / K} - \mu_{A_{1} / K}
			=
				\chi^{0}(X, \mathcal{N}_{f}).
		\]
\end{theorem}

\begin{proof}
	We first show the first equality in \eqref{0009}.
	The object $\alg{H}^{i}(X, \mathcal{A}_{j})$ is \'etale for $i \ne 0, 1$
	by \cite[Theorem 3.4.1 (1) and (2)]{Suz20a}.
	The object $\alg{H}^{0}(X, \mathcal{A}_{j})$
	is an extension of an \'etale group with finitely generated group of geometric points
	by (the perfection of) an abelian variety over $k$
	by \cite[Theorem 3.4.1 (2)]{Suz20a}.
	Since $A_{1}$ and $A_{2}$ are isogenous,
	the objects $\alg{H}^{0}(X, \mathcal{A}_{1})$
	and $\alg{H}^{0}(X, \mathcal{A}_{2})$ are isomorphic up to bounded torsion.
	Hence, their abelian variety parts are isogenous.
	Combining these, we get the first equality in \eqref{0009}.
	
	Let $\mathcal{A}_{2}' = \mathcal{A}_{1} / \mathcal{N}_{f}$.
	Then we have the diagram \eqref{0008}.
	Let $\mathcal{D}_{f} = [\mathcal{A}_{2}' \to \mathcal{A}_{2}]$
	be the mapping cone of the morphism $\mathcal{A}_{2}' \to \mathcal{A}_{2}$.
	We have distinguished triangles
		\begin{equation} \label{0040}
				R \alg{\Gamma}(X, \mathcal{N}_{f})
			\to
				R \alg{\Gamma}(X, \mathcal{A}_{1})
			\to
				R \alg{\Gamma}(X, \mathcal{A}_{2}'),
		\end{equation}
		\begin{equation} \label{0029}
				R \alg{\Gamma}(X, \mathcal{A}_{2}')
			\to
				R \alg{\Gamma}(X, \mathcal{A}_{2})
			\to
				R \alg{\Gamma}(X, \mathcal{D}_{f})
		\end{equation}
	in $D(k^{\ind\rat}_{\pro\et})$.
	
	We will calculate $R \alg{\Gamma}(X, \mathcal{D}_{f})$.
	Let $v \in X_{0}$.
	By \cite[Proposition 3.4.2 (a)]{Suz20b} and its proof,
	we know that $\alg{H}^{i}(\Order_{v}, \mathcal{A}_{2}') = 0$ for $i \ne 0$
	and $\alg{H}^{0}(\Order_{v}, \mathcal{A}_{2}')$ is (the perfection of)
	the Greenberg transform $\Gr_{\Order_{v}} \mathcal{A}_{2}'$
	of $\mathcal{A}_{2}' \times_{X} \Spec \Order_{v}$.
	The same is true with $\mathcal{A}_{2}'$ replaced by $\mathcal{A}_{2}$.
	Hence, the distinguished triangle
		\[
				R \alg{\Gamma}(\Order_{v} / k, \mathcal{A}_{2}')
			\to
				R \alg{\Gamma}(\Order_{v} / k, \mathcal{A}_{2})
			\to
				R \alg{\Gamma}(\Order_{v} / k, \mathcal{D}_{f})
		\]
	reduces to the short exact sequence
		\[
				0
			\to
				\Res_{k(v) / k}(\Gr_{\Order_{v}} \mathcal{A}_{2}')
			\to
				\Res_{k(v) / k}(\Gr_{\Order_{v}} \mathcal{A}_{2})
			\to
				\Res_{k(v) / k} D_{f, v}
			\to
				0
		\]
	coming from \eqref{0011},
	where $\Res_{k(v) / k}$ is the Weil restriction functor.
	In particular,
	$\alg{H}^{0}(\Order_{v} / k, \mathcal{D}_{f}) \in \Alg_{\uc} / k$
	and
		\begin{equation} \label{0013}
				\dim
				\alg{H}^{0}(\Order_{v} / k, \mathcal{D}_{f})
			=
				[k(v) : k] \dim D_{f, v}.
		\end{equation}
	
	Let $U \subset X$ be a dense open subscheme
	where $A_{1}$ and hence also $A_{2}$ have good reduction.
	Then $\Tilde{f} \colon \mathcal{A}_{1} \to \mathcal{A}_{2}$
	restricted to $U$ is an isogeny of abelian schemes.
	Therefore, the restriction of $\mathcal{D}_{f}$ to $U$ is zero.
	Hence, by \eqref{0039},
	we have a commutative diagram of distinguished triangles
		\[
			\begin{CD}
					R \alg{\Gamma}_{c}(U, \mathcal{A}_{2}')
				@> \sim >>
					R \alg{\Gamma}_{c}(U, \mathcal{A}_{2})
				@>>>
					0
				\\ @VVV @VVV @VVV \\
					R \alg{\Gamma}(X, \mathcal{A}_{2}')
				@>>>
					R \alg{\Gamma}(X, \mathcal{A}_{2})
				@>>>
					R \alg{\Gamma}(X, \mathcal{D}_{f})
				\\ @VVV @VVV @VV \wr V \\
					\displaystyle
					\bigoplus_{v \in X \setminus U}
						R \alg{\Gamma}(\Order_{v} / k, \mathcal{A}_{2}')
				@>>>
					\displaystyle
					\bigoplus_{v \in X \setminus U}
						R \alg{\Gamma}(\Order_{v} / k, \mathcal{A}_{2})
				@>>>
					\displaystyle
					\bigoplus_{v \in X \setminus U}
						R \alg{\Gamma}(\Order_{v} / k, \mathcal{D}_{f}).
			\end{CD}
		\]
	Hence, $R \alg{\Gamma}(X, \mathcal{D}_{f})$ has cohomology only in degree zero
	that is an object of $\Alg_{\uc} / k$ and
		\begin{equation} \label{0041}
				\dim \alg{H}^{0}(X, \mathcal{D}_{f})
			=
				\sum_{v \in X \setminus U}
					\dim \alg{H}^{0}(\Order_{v} / k, \mathcal{D}_{f})
			=
				c(f)
		\end{equation}
	by \eqref{0013}.
	
	With the properties of $\alg{H}^{i}(X, \mathcal{A}_{2})$
	noted before Proposition \ref{0038},
	we know that
	$\alg{H}^{i}(X, \mathcal{A}_{2}')$ is in $\Loc \Alg / k$ for all $i$
	and zero for almost all $i$.
	The distinguished triangle \eqref{0029} and the equality \eqref{0041}
	then imply that
		\begin{equation} \label{0015}
				\chi^{0}(X, \mathcal{A}_{2})
			=
					\chi^{0}(X, \mathcal{A}_{2}')
				+
					\dim \alg{H}^{0}(X, \mathcal{D}_{f})
			=
					\chi^{0}(X, \mathcal{A}_{2}')
				+
					c(f).
		\end{equation}
	The distinguished triangle \eqref{0040} implies that
		\begin{equation} \label{0014}
				\chi^{0}(X, \mathcal{A}_{1})
			=
				\chi^{0}(X, \mathcal{N}_{f})
			+
				\chi^{0}(X, \mathcal{A}_{2}').
		\end{equation}
	Combining the first equality in \eqref{0009}, \eqref{0015} and \eqref{0014},
	we get the result.
\end{proof}


\section{%
	\texorpdfstring{Change in $\mu$ by flat Euler characteristics: finite base field case}
	{Change in mu by flat Euler characteristics: finite base field case}
}

Let $X$ be a geometrically connected, proper, smooth curve
over a perfect field $k$ of characteristic $p > 0$.
Let $K$ be its function field.
Assume that $k$ is finite with $q$ elements in this section.
We will deduce a formula specific to a finite base field
from Theorem \ref{0003}.

Let $A$ be an abelian variety over $K$.
In this finite base field case,
by \cite[Theorem 2.4.3]{LLSTT21},
the number $\mu_{A / K}$ defined in \eqref{0036}
agrees with the $\mu$-invariant of $A$
as defined in the second paragraph of Section \ref{0079}.

For an fppf sheaf $\mathcal{F}$ on $X$,
let
	\[
			\chi(X, \mathcal{F})
		=
			\prod_{i}
				(\# H^{i}(X, \mathcal{F}))^{(-1)^{i}}
	\]
if $H^{i}(X, \mathcal{F})$ is finite for all $i$
and zero for almost all $i$.
The following proposition generalizes
\cite[Chapter III, Lemma 8.9]{Mil06}
on finite flat group schemes
to quasi-finite ones:

\begin{proposition} \label{0049}
	For any separated, quasi-finite, flat group scheme $\mathcal{N}$ over $X$,
	the groups $H^{i}(X, \mathcal{N})$ are finite for all $i$
	and zero for almost all $i$.
\end{proposition}

\begin{proof}
	As noted after \eqref{0065},
	we have
		\[
				R \Gamma(X, \mathcal{N})
			\cong
				R \Gamma(k, R \alg{\Gamma}(X, \mathcal{N})).
		\]
	Hence, the proposition follows from Proposition \ref{0038}
	by applying Proposition \ref{0005} \eqref{0001} to
	$G = R \alg{\Gamma}(X, \mathcal{N})$.
\end{proof}

\begin{theorem} \label{0004}
	Let $f \colon A_{1} \to A_{2}$ be an isogeny of abelian varieties over $K$.
	Let $\mathcal{N}_{f}$ be the schematic closure
	of the kernel of $f$ in the N\'eron model of $A_{1}$ over $X$.
	Then we have
		\[
				\mu_{A_{2} / K} - \mu_{A_{1} / K}
			=
					v_{q}(\chi(X, \mathcal{N}_{f}))
				-
					c(f).
		\]
	In particular, if $f$ extends to an isogeny on the N\'eron models, then
		\[
				\mu_{A_{2} / K} - \mu_{A_{1} / K}
			=
				v_{q}(\chi(X, \mathcal{N}_{f})).
		\]
\end{theorem}

\begin{proof}
	This follows from Theorem \ref{0003}
	by applying Proposition \ref{0005} \eqref{0002}
	to $G = R \alg{\Gamma}(X, \mathcal{N}_{f})$.
\end{proof}


\section{Lie complexes}

Our next task is to deduce ``linear'' (or coherent) versions of
Theorems \ref{0003} and \ref{0004},
which will be completed in Section \ref{0046}.
As preparations, in this section,
we will define the linearized terms needed for this,
which are derived versions of Lie algebras
(relative tangent spaces at the zero section).

Let $X$ be a geometrically connected, proper, smooth curve
over a perfect field $k$ of characteristic $p > 0$.
Let $K$ be its function field.
Let $U \subset X$ be a dense open subscheme.
For a flat group scheme $G$ locally of finite type over $U$
with identity section $e \colon U \into G$,
let $L_{G / U} \in D(\Order_{G})$ be the cotangent complex of $G / U$
(\cite[Chapter II, (1.2.7.1)]{Ill71}).
Let $e^{\ast}$ be the pullback functor for sheaves of $\Order_{G}$-modules.
Set
	\[
			l_{G}
		:=
			L e^{\ast} L_{G / U}
		\in
			D(\Order_{U})
	\]
as in \cite[Chapter VII, (3.1.1.1)]{Ill72},
which is (represented by) a perfect complex of $\Order_{U}$-modules
with tor-amplitude in $[-1, 0]$ by \cite[Chapter VII, (3.1.1.3)]{Ill72}.
Let
	\[
			R \Lie G
		:=
			l_{G}^{\vee}
		\in
			D^{b}(\Order_{U})
	\]
be the Lie complex defined as the derived $\Order_{U}$-dual of $l_{G}$
(\cite[Chapter VII, (3.1.1.1)$'$]{Ill72}).
It is (represented by) a perfect complex of $\Order_{U}$-modules
with tor-amplitude in $[0, 1]$
(\cite[Chapter VII, (3.1.1.3)$'$]{Ill72})
whose zeroth cohomology is the usual Lie algebra $\Lie G$ of $G$.
If $0 \to G' \to G \to G'' \to 0$ is an exact sequence
of flat group schemes locally of finite type over $U$,
then we have a canonical distinguished triangle
	\begin{equation} \label{0069}
			R \Lie G'
		\to
			R \Lie G
		\to
			R \Lie G''
	\end{equation}
by \cite[Chapter VII, Proposition 3.1.1.5]{Ill72}.
If $G$ is smooth over $U$, then $R \Lie G$ is
the usual Lie algebra $\Lie G$ of $G$.
Hence, if a flat group scheme $N$ locally of finite type over $U$
can be embedded into a smooth group scheme $G$ as a closed group subscheme
with quotient $H = G / N$
(which is a smooth group scheme by \cite[Theorem 4.C]{Ana73}),
then
	\begin{equation} \label{0042}
			R \Lie N
		\cong
			[\Lie G \to \Lie H][-1].
	\end{equation}
Such an embedding always exists
if we assume the quasi-compactness of $N$ and restrict all the objects to $K$
by \cite[Chapter III, Section 5, Lemma 6.5]{DG70}
or if $N$ is finite by \cite[Chapter III, Theorem A.5]{Mil06}.
Hence, in these cases, one may take this formula as the definition of $R \Lie N$.

This object $R \Lie N$ is particularly simple
if the generic fiber of $N$ is \'etale:

\begin{proposition} \label{0044}
	Let $\mathcal{N}$ be a separated, quasi-finite, flat group scheme over $U$
	whose generic fiber is \'etale.
	Let $e \colon U \to \mathcal{N}$ be its zero section.
	Then the derived $\Order_{U}$-dual of $R \Lie \mathcal{N}$
	(namely, $l_{\mathcal{N}}$) is isomorphic to
	$e^{\ast} \Omega^{1}_{\mathcal{N} / U}$
	placed in degree zero.
\end{proposition}

\begin{proof}
	$H^{-1}(l_{\mathcal{N}})$ is torsion-free
	by the tor-amplitude condition.
	It is torsion by the \'etaleness assumption.
	Hence, $H^{-1}(l_{\mathcal{N}}) = 0$.
	The description of $H^{0}(l_{\mathcal{N}})$ directly follows
	from the definition of the cotangent complex.
\end{proof}

For a perfect complex of $\Order_{U}$-modules $\mathcal{F}$,
let $\det \mathcal{F}$ be its determinant invertible sheaf
(\cite[Tag 0FJW]{Sta26}).
When $U = X$,
let $\deg \mathcal{F} := \deg(\det \mathcal{F})$,
where $\deg$ on the right-hand side denotes
the degree of an invertible sheaf with respect to $k$.%
\footnote{\label{0076}
	As noted in the paragraph after \cite[0AYR]{Sta26},
	the definition of the degree of an invertible sheaf
	over a $1$-dimensional proper $k$-scheme depends on $k$.
	The definition of $\deg \mathcal{F}$ used in \cite[0AYR]{Sta26}
	(when $\mathcal{F}$ is concentrated in degree zero)
	is $\chi(X, \mathcal{F}) - \chi(K, \mathcal{F}|_{K}) \chi(X, \Order_{X})$,
	with which the ``Riemann--Roch theorem'' below is true by definition.
	For the agreement of this definition with the ``usual'' definition
	(which is the real content of the Riemann--Roch theorem),
	see \cite[Tags 0AYX, 0DJ5 and 0AYY]{Sta26}.
}
Define
	\[
			\chi(X, \mathcal{F})
		:=
			\sum_{i}
				(-1)^{i}
				\dim_{k} H^{i}(X, \mathcal{F}).
	\]
	\[
			\chi(K, \mathcal{F}|_{K})
		:=
			\sum_{i}
				(-1)^{i}
				\dim_{K} H^{i}(K, \mathcal{F}|_{K}),
	\]
where $H^{i}(K, \mathcal{F}|_{K})$ is equal to
the dimension of the $i$-th cohomology object
$H^{i}(\mathcal{F}|_{K})$ as a $K$-vector space.
Then
	\[
			\chi(X, \mathcal{F})
		=
				\chi(K, \mathcal{F}|_{K})
				\chi(X, \Order_{X})
			+
				\deg \mathcal{F}
	\]
by the Riemann--Roch theorem.

\begin{proposition} \label{0043}
	Let $\mathcal{N}$ be a separated, quasi-finite, flat group scheme over $X$.
	Then
		\[
				\chi(X, R \Lie \mathcal{N})
			=
				\deg(R \Lie \mathcal{N}).
		\]
\end{proposition}

\begin{proof}
	Let $N$ be the generic fiber of $\mathcal{N}$.
	It is enough to show that $\chi(K, R \Lie N) = 0$.
	This follows from \eqref{0042}
	since $\dim G = \dim H$ in this case.
\end{proof}


\section{Duality of Euler characteristics}

Let $X$ be a geometrically connected, proper, smooth curve
over a perfect field $k$ of characteristic $p > 0$.
Let $K$ be its function field.
We need duality results for $\deg(R \Lie \mathcal{N})$
and $\chi^{0}(X, \mathcal{N})$.

\begin{proposition} \label{0023}
	Let $U \subset X$ be a dense open subscheme.
	Let $N$ be a finite flat group schemes over $U$
	with Cartier dual $M$.
	Then there exists a canonical isomorphism
	$\det(R \Lie N) \cong (\det(R \Lie M))^{\otimes -1}$
	of $\Order_{U}$-modules.
\end{proposition}

\begin{proof}
	First assume that $N$ can be embedded into an abelian scheme
	as a closed group subscheme.
	Then we have a resolution
	$0 \to N \to A_{1} \to A_{2} \to 0$
	by abelian schemes over $U$.
	Denote its dual (\cite[Theorem (19.1)]{Oor66}) by
	$0 \to M \to B_{2} \to B_{1} \to 0$.
	Let $d$ be the relative dimension of $A_{1}$ (or $A_{2}$).
	Let $\pi_{i} \colon A_{i} \to U$ be the structure morphism.
	By $\Lie B_{i} \cong R^{1} \pi_{i, \ast} \Order_{A_{i}}$
	and \cite[Chapter VII, Section 4.21, Theorem 10]{Ser88},
	we have
		\[
				\det(\Lie B_{i})
			\cong
				\det(R^{1} \pi_{i, \ast} \Order_{A_{i}})
			\cong
				R^{d} \pi_{i, \ast} \Order_{A_{i}},
		\]
	which is dual to $\pi_{i, \ast} \Omega_{A_{i} / U}^{d}$
	by relative Serre duality
	(\cite[Tag 0AU3 (4) and (9)]{Sta26}
	with $\Order_{U}$ taken as a dualizing complex for $U$).
	This sheaf $\pi_{i, \ast} \Omega_{A_{i} / U}^{d}$ is dual to
	$\det(\Lie A_{i})$.
	Hence, $\det(\Lie B_{i}) \cong \det(\Lie A_{i})$ canonically.
	Therefore,
			\begin{equation} \label{0021}
			\begin{split}
				&	\det(R \Lie N)
				\cong
						\det(\Lie A_{1})
					\tensor
						(\det(\Lie A_{2}))^{\tensor -1}
				\\
				& \quad
					\cong
						\det(\Lie B_{1})
					\tensor
						(\det(\Lie B_{2}))^{\tensor -1}
				\cong
					(\det(R \Lie M))^{\otimes -1}.
			\end{split}
			\end{equation}
	If we have another resolution
	$0 \to N \to A_{1}' \to A_{2}' \to 0$
	and a commutative diagram
		\[
			\begin{CD}
					0
				@>>>
					N
				@>>>
					A_{1}'
				@>>>
					A_{2}'
				@>>>
					0
				\\ @. @| @VVV @VVV \\
					0
				@>>>
					N
				@>>>
					A_{1}
				@>>>
					A_{2}
				@>>>
					0
			\end{CD}
		\]
	with $A_{1}' \to A_{1}$ smooth with connected fibers,
	then the kernels of $A_{1}' \to A_{1}$ and $A_{2}' \to A_{2}$ are isomorphic.
	Hence, the resulting isomorphism
			\begin{equation} \label{0022}
			\begin{split}
				&	\det(R \Lie N)
				\cong
						\det(\Lie A_{1}')
					\tensor
						(\det(\Lie A_{2}'))^{\tensor -1}
				\\
					& \quad
					\cong
						\det(\Lie B_{1}')
					\tensor
						(\det(\Lie B_{2}'))^{\tensor -1}
				\cong
					(\det(R \Lie M))^{\otimes -1}
			\end{split}
			\end{equation}
	is compatible with \eqref{0021}.
	If we just have another resolution
	$0 \to N \to A_{1}' \to A_{2}' \to 0$
	(no commutative diagram),
	then take $A_{1}'' = A_{1} \times A_{1}'$
	and the diagonal embedding $N \into A_{1}''$,
	defining $A_{2}'' = A_{1}'' / N$.
	Then we have a commutative diagram
		\[
			\begin{CD}
					0
				@>>>
					N
				@>>>
					A_{1}'
				@>>>
					A_{2}'
				@>>>
					0
				\\ @. @| @AAA @AAA \\
					0
				@>>>
					N
				@>>>
					A_{1}''
				@>>>
					A_{2}''
				@>>>
					0
				\\ @. @| @VVV @VVV \\
					0
				@>>>
					N
				@>>>
					A_{1}
				@>>>
					A_{2}
				@>>>
					0.
			\end{CD}
		\]
	Hence, \eqref{0022} is still compatible with \eqref{0021}.
	This means that we have the required canonical isomorphism
	in this (embeddable) case.
	
	In general, such an embedding exists Zariski locally on $U$
	by \cite[Theorem 3.1.1]{BBM82}.
	Now the canonicity of the isomorphism in the embeddable case proved above
	allows us to patch the local isomorphisms.
\end{proof}

\begin{proposition} \label{0025}
	Let $N$ be a finite flat group scheme over $X$
	with Cartier dual $M$.
	Then
		\[
			\deg(R \Lie N) = - \deg(R \Lie M).
		\]
\end{proposition}

\begin{proof}
	This follows from Propositions \ref{0023}.
\end{proof}

\begin{proposition} \label{0026}
	Let $N$ be a finite flat group scheme over $X$
	with Cartier dual $M$.
	Then
		\[
			\chi^{0}(X, N) = - \chi^{0}(X, M).
		\]
\end{proposition}

\begin{proof}
	Let $\sheafhom_{k^{\ind\rat}_{\pro\et}}$ be
	the sheaf-Hom functor for $\Ab(k^{\ind\rat}_{\pro\et})$.
	Denote the functor
	$R \sheafhom_{k^{\ind\rat}_{\pro\et}}(\var, \Z)$ by $(\var)^{\SDual}$
	(the Serre dual functor).
	Then there exists a canonical isomorphism
		\[
				R \alg{\Gamma}(X, N)
			\cong
				R \alg{\Gamma}(X, M)^{\SDual}[-1]
		\]
	by \cite[Theorem 3.1.3]{Suz20a}.
	Both objects $R \alg{\Gamma}(X, N)$ and $R \alg{\Gamma}(X, M)$ are
	in $D^{b}(\Alg_{\uc} / k)$
	by Proposition \ref{0038}.
	Let $K_{0}(D^{b}(\Alg_{\uc} / k))$ be
	the Grothendieck group of $D^{b}(\Alg_{\uc} / k)$.
	Let
		\[
				\chi
			\colon
				K_{0}(D^{b}(\Alg_{\uc} / k))
			\to
				\Z
		\]
	be the unique homomorphism that assigns the dimension to an object of $\Alg_{\uc} / k$.
	Then
		\[
				\chi(R \alg{\Gamma}(X, M)^{\SDual})
			=
				\chi(R \alg{\Gamma}(X, M))
		\]
	by \cite[Proposition 5.3 (2)]{OS26}.
	As
		\[
					\chi^{0}(X, N)
				=
					\chi(R \alg{\Gamma}(X, N)),
			\quad
					\chi^{0}(X, M)
				=
					\chi(R \alg{\Gamma}(X, M)),
		\]
	the result follows.
\end{proof}


\section{%
	\texorpdfstring{Change in $\mu$ by degrees of Lie complexes}
	{Change in mu by degrees of Lie complexes}
}
\label{0046}

Let $X$ be a geometrically connected, proper, smooth curve
over a perfect field $k$ of characteristic $p > 0$.
Let $K$ be its function field.
Now we can interpret $\chi(X, \mathcal{N})$
in terms of $\deg(R \Lie \mathcal{N})$,
which gives a linearized version Theorems \ref{0003} and \ref{0004}.

\begin{theorem} \label{0027}
	Let $\mathcal{N}$ be a separated, quasi-finite, flat group scheme over $X$.
	Then
		\[
				\chi^{0}(X, \mathcal{N})
			=
				\deg(R \Lie \mathcal{N}).
		\]
\end{theorem}

\begin{proof}
	If the multiplication by $p$ is invertible on $\mathcal{N}$,
	then $\mathcal{N}$ is \'etale over $X$.
	Hence, $\chi^{0}(X, \mathcal{N}) = 0$ by Proposition \ref{0064}
	and $R \Lie \mathcal{N} = 0$, giving the desired equality in this case.
	Therefore, we may assume that $\mathcal{N}$ is killed by a power of $p$.
	
	Let $G$ be any vector bundle on $X$ (viewed as a vector group).
	We show that
		\begin{equation} \label{0024}
				\chi^{0}(X, G)
			=
				\chi(X, \Lie G).
		\end{equation}
	The natural $k$-linear structures on $G$ and $\Ga$ gives a morphism
		\begin{equation} \label{0066}
				R \Gamma(X, G) \tensor_{k} \Ga
			\to
				R \alg{\Gamma}(X, G).
		\end{equation}
	For any $i \in \Z$, the morphism induced on the $i$-th cohomology takes the form
		\[
				H^{i}(X, G) \tensor_{k} \Ga
			\to
				\alg{H}^{i}(X, G).
		\]
	The left-hand and right-hand sides are
	the pro-\'etale sheafifications of the presheaves
		\[
					k'
				\mapsto
					H^{i}(X, G) \tensor_{k} k'
			\quad \text{and} \quad
					k'
				\mapsto
					H^{i}(X \times_{k} k', G),
		\]
	respectively, where $k'$ runs over ind-rational $k$-algebras.
	We have an isomorphism
		\[
				H^{i}(X, G) \tensor_{k} k'
			\isomto
				H^{i}(X \times_{k} k', G)
		\]
	by the flat base change theorem for coherent cohomology.
	Hence, \eqref{0066} is an isomorphism, giving \eqref{0024}.
	
	Assume that $\mathcal{N}$ is finite flat
	with Cartier dual of height one.
	Then we have an exact sequence
	$0 \to \mathcal{N} \to G \to H \to 0$ over $X$
	with $G$ and $H$ vector groups
	by \cite[Chapter III, Theorem 5.1]{Mil06}.
	Hence, \eqref{0067}, \eqref{0024} and Proposition \ref{0043}
	give the result in this case.
	
	Assume next that the generic fiber of $\mathcal{N}$ is \'etale.
	Let $U \subset X$ be dense open where $\mathcal{N}$ is \'etale.
	Then $R \alg{\Gamma}_{c}(U, \mathcal{N})$ has finite \'etale cohomologies
	by Proposition \ref{0064}
	and $R \alg{\Gamma}_{c}(U, R \Lie \mathcal{N}) = 0$.
	Hence, it is enough to show that
		\[
				\chi^{0}(\Order_{v}, \mathcal{N})
			=
				\chi(\Order_{v}, R \Lie \mathcal{N})
		\]
	for each $v \in X \setminus U$
	(where the left-hand side is the alternating sum of
	the dimensions of $\alg{H}^{i}(\Order_{v}, \mathcal{N})$
	and the right-hand side is the alternating sum
	of the dimensions of $H^{i}(\Order_{v}, R \Lie \mathcal{N})$ over $k(v)$).
	This equality follows from
		\[
				\chi^{0}(\Order_{v}, \mathcal{N})
			=
				- \dim \alg{H}^{1}(\Order_{v}, \mathcal{N})
			=
				- \length_{\Order_{v}}
					e^{\ast}
					\Omega^{1}_{\mathcal{N} / \Order_{v}}
			=
				\chi(\Order_{v}, R \Lie \mathcal{N})
		\]
	by \cite[Propositions 3.2 and 3.7]{OS26}
	and Proposition \ref{0044},
	where $e$ denotes the zero section for $\mathcal{N}$.
	
	In general, let $N$ be the generic fiber of $\mathcal{N}$.
	Let
		\[
			0 \subset N^{m} \subset N^{0} \subset N
		\]
	be the multiplicative part and the identity component, respectively.
	Let
		\[
			0 \subset \mathcal{N}^{m} \subset \mathcal{N}^{0} \subset \mathcal{N}
		\]
	be their schematic closures in $\mathcal{N}$.
	Then $\mathcal{N}^{m}$ and $\mathcal{N}^{0}$ are finite flat
	by the argument in the proof of Proposition \ref{0038},
	and the subquotients of this filtration are all
	separated, quasi-finite and flat.
	Since both sides of the desired equality are additive in short exact sequences
	by \eqref{0067} and \eqref{0069},
	it is enough to check the statement for each subquotient.
	The statement for $\mathcal{N}^{m}$ reduces to
	the statement for its Cartier dual
	by Propositions \ref{0025} and \ref{0026},
	which has \'etale generic fiber.
	Hence, this case is done.
	For the statement for $\mathcal{N}^{0} / \mathcal{N}^{m}$,
	its Cartier dual is infinitesimal,
	so it has a filtration with height one subquotients,
	which is done.
	The quotient $\mathcal{N} / \mathcal{N}^{0}$ has \'etale generic fiber,
	so it is done.
\end{proof}

The following answers \cite[Chapter III, Problem 8.10]{Mil06}.

\begin{theorem} \label{0030}
	Assume that $k$ is finite with $q$ elements.
	Let $\mathcal{N}$ be a separated, quasi-finite, flat group scheme over $X$.
	Then
		\[
				\chi(X, \mathcal{N})
			=
				q^{\deg(R \Lie \mathcal{N})}.
		\]
\end{theorem}

\begin{proof}
	By Proposition \ref{0038},
	we may apply Proposition \ref{0005} \eqref{0002}
	to $G = R \alg{\Gamma}(X, \mathcal{N})$.
	Since $R \Gamma(k, G) \cong R \Gamma(X, \mathcal{N})$
	by \eqref{0065},
	the desired equality follows from Theorem \ref{0027}.
\end{proof}

The following gives a function field version of
the ``Isogeny formula (First form)'' at the end of \cite{Sch87}.

\begin{theorem} \label{0031}
	Let $f \colon A_{1} \to A_{2}$ be an isogeny of abelian varieties over $K$.
	Let $\mathcal{N}_{f}$ be the schematic closure
	of the kernel of $f$ in $\mathcal{A}_{1}$.
	Then we have
		\[
				\mu_{A_{2} / K} - \mu_{A_{1} / K}
			=
					\deg(R \Lie \mathcal{N}_{f})
				-
					c(f).
		\]
	In particular, if $f$ extends to an isogeny on the N\'eron models, then
		\[
				\mu_{A_{2} / K} - \mu_{A_{1} / K}
			=
				\deg(R \Lie \mathcal{N}_{f}).
		\]
\end{theorem}

\begin{proof}
	This follows from Theorems \ref{0027} and \ref{0003}.
\end{proof}


\section{Isogeny invariance of the BSD formula}

As an application of our methods and Theorem \ref{0030},
we will prove that the validity of the Birch--Swinnerton-Dyer conjecture
for abelian varieties over function fields
is invariant under an arbitrary isogeny,
whose degree may be divisible by the characteristic $p$.
For this, it is convenient to use the Weil-\'etale version of the conjecture
given in \cite{GS20}.

We first recall Weil-\'etale cohomology over finite fields from
\cite[Section 5]{GS20} and \cite[Section 10]{Suz20b}
in the form convenient in our setting.
Let $k$ be a finite field with $q$ elements.
Let $F_{k} \colon \closure{k} \isomto \closure{k}$ be
the $q$-th power map.
Note that $\Gamma(\closure{k}, \var)$ is an exact functor
on $\Ab(k^{\ind\rat}_{\pro\et})$
and so $R \Gamma(\closure{k}, \var) = \Gamma(\closure{k}, \var)$.
For any $G \in D(k^{\ind\rat}_{\pro\et})$,
define $R \Gamma_{W}(k, G) \in D(\Ab)$ to be
the canonical mapping fiber of
$F_{k} - 1 \colon \Gamma(\closure{k}, G) \to \Gamma(\closure{k}, G)$,
so that it fits in a canonical distinguished triangle
	\[
			R \Gamma_{W}(k, G)
		\to
			\Gamma(\closure{k}, G)
		\stackrel{F_{k} - 1}{\longrightarrow}
			\Gamma(\closure{k}, G).
	\]
This is the group cohomology of the infinite cyclic group generated by $F_{k}$
with coefficients in $\Gamma(\closure{k}, G)$.
Let $H_{W}^{i}(k, G)$ be the $i$-th cohomology of $R \Gamma_{W}(k, G)$.

Below we only use this construction when $G$ is bounded below
and the $i$-th cohomology object $H^{i}(G)$ for any $i$
as a functor on ind-rational $k$-algebras
commutes with filtered direct limits
(which implies that pro-\'etale cohomology with coefficients in $H^{i}(G)$
agrees with \'etale cohomology with coefficients in $H^{i}(G)$
as in the proof of Proposition \ref{0005}).
Assume these conditions on $G$.
Then the above $R \Gamma_{W}(k, G)$ agrees with
the object $R \Gamma(k_{W}, G)$ in
\cite[Section 10, the paragraphs before Proposition (10.3)]{Suz20b}
and with the object $R \Gamma_{W}(k, R \nu_{\ast} G)$ in
\cite[Section 5]{GS20},
where $\nu \colon \Spec k^{\ind\rat}_{\pro\et} \to \Spec k_{\et}$
is the morphism of sites defined by the inclusion functor.
In particular, we have a canonical distinguished triangle
	\begin{equation} \label{0071}
			R \Gamma(k, G)
		\to
			R \Gamma_{W}(k, G)
		\to
			R \Gamma(k, G) \tensor \Q[-1]
	\end{equation}
by \cite[Section 5, Equation (6)]{GS20}
(or more precisely \cite[Corollary 5.2]{Gei04}).
Since
$F_{k} - 1 \colon \Gamma(\closure{k}, G) \to \Gamma(\closure{k}, G)$
is zero if $G = \Z$,
we have $R \Gamma_{W}(k, \Z) \cong \Z \oplus \Z[-1]$.
Let $e \in H_{W}^{1}(k, \Z)$ be $1 \in \Z$ in this isomorphism.
Since $e \cup e \in H_{W}^{2}(k, \Z) = 0$
(where the cup product is for group cohomology of the infinite cyclic group),
we have a complex
	\[
			\cdots
		\stackrel{e}{\to}
			H_{W}^{i}(k, G)
		\stackrel{e}{\to}
			H_{W}^{i + 1}(k, G)
		\stackrel{e}{\to}
			\cdots
	\]
given by cup products with $e$.
This complex $(H_{W}^{\ast}(k, G), e)$ has torsion cohomologies
as noted in the second paragraph of \cite[Section 8]{GS20}
(by \cite[Corollary 5.2]{Gei04}).
When $H_{W}^{i}(k, G)$ is finitely generated for all $i$
and zero for almost all $i$,
the complex $(H_{W}^{\ast}(k, G), e)$ thus has finite cohomologies,
so that we can define
	\[
			\chi_{W}(k, G)
		:=
			\prod_{i}
				\bigl(
					\# H^{i}(H_{W}^{\ast}(k, G), e)
				\bigr)^{(-1)^{i}}
	\]
(which is denoted by $\chi(H_{W}^{\ast}(k, G), e)$
in the notation of \cite[Section 8]{GS20}).
If $G$ is bounded, then $H_{W}^{i}(k, G)$ is zero for almost all $i$.
If $G', G'' \in D(k^{\ind\rat}_{\pro\et})$ satisfy the same conditions as $G$
and $G' \to G \to G''$ is a distinguished triangle,
we have a bounded long exact sequence
	\[
			\cdots
		\to
			H_{W}^{i}(k, G')
		\to
			H_{W}^{i}(k, G)
		\to
			H_{W}^{i}(k, G'')
		\to
			\cdots
	\]
of finitely generated abelian groups compatible with the actions of $e$.
In this situation, the argument in
\cite[Proof of Theorem 7.1; Paragraph after Equation (15)]{Gei06}
shows that
	\begin{equation} \label{0070}
			\chi_{W}(k, G)
		=
			\chi_{W}(k, G')
			\chi_{W}(k, G'').
	\end{equation}

Let $X$ be a geometrically connected, proper, smooth curve over $k$
(finite with $q$ elements as above).
Let $K$ be its function field.
For $G \in D(X_{\fppf})$,
define
	\[
			R \Gamma_{W}(X, G)
		:=
			R \Gamma_{W}(k, R \alg{\Gamma}(X, G)).
	\]
It fits in a distinguished triangle
	\[
			R \Gamma_{W}(X, G)
		\to
			R \Gamma(X \times_{k} \closure{k}, G)
		\stackrel{F_{k} - 1}{\longrightarrow}
			R \Gamma(X \times_{k} \closure{k}, G).
	\]
Let $H_{W}^{i}(X, G)$ be the $i$-th cohomology of $R \Gamma_{W}(X, G)$.

We only use this construction
when $G$ is a bounded complex of group schemes locally of finite type.
In this case, $\alg{H}^{i}(X, G)$ for any $i$
as a functor on ind-rational $k$-algebras
commutes with filtered direct limits by \cite[Proposition 2.7.8]{Suz20a}
and hence the above $R \Gamma_{W}(X, G)$ agrees with the object
$R \Gamma_{W}(X, R \varepsilon_{\ast} G)$ in \cite[Section 5]{GS20},
where $\varepsilon \colon X_{\fppf} \to X_{\et}$
is the morphism of sites defined by the inclusion functor.
When $H_{W}^{i}(X, G)$ is finitely generated for all $i$
and zero for almost all $i$,
we define $\chi_{W}(X, G) := \chi_{W}(k, R \alg{\Gamma}(X, G))$.

Let $A / K$ be an abelian variety with N\'eron model $\mathcal{A} / X$.
Let $r$ be the rank of $A(K)$.
Let $L(A, t)$ be the Hasse--Weil $L$-function of $A / K$
as a rational function of $t$.
We recall the Weil-\'etale BSD formula from
\cite[Theorem 1.1, Proposition 8.4]{GS20}.
We saw at the beginning of Section \ref{0037}
that $R \alg{\Gamma}(X, G)$ is bounded.
Assume that the Tate--Shafarevich group $\Sha(A)$ is finite.
Then \cite[Theorem 1.1, Proposition 8.4]{GS20} says that
$H_{W}^{i}(X, \mathcal{A})$ is finitely generated for all $i$ and we have
	\begin{equation} \label{0028}
			\lim_{t \to q^{-1}}
				\frac{
					L(A, t)
				}{
					(1 - q t)^{r}
				}
		=
			\frac{
				q^{\chi(X, \Lie \mathcal{A})}
			}{
				\chi_{W}(X, \mathcal{A})
			}.
	\end{equation}
This formula (under the assumption of the finiteness of $\Sha(A)$) is proved by
\cite{KT03}.
Here we prove, without using the result of \cite{KT03},
that the validity of this formula is invariant under an arbitrary isogeny on $A$.
This is \cite[Chapter III, Problem 9.10]{Mil06}.

\begin{theorem} \label{0051}
	Let $f \colon A_{1} \to A_{2}$ be an isogeny of abelian varieties over $K$.
	Assume that $\Sha(A_{1})$ and $\Sha(A_{2})$ are finite.%
	\footnote{
		The finiteness of $\Sha(A_{1})$ is equivalent to the finiteness of $\Sha(A_{2})$
		as essentially mentioned in
		\cite[Chapter I, Remark 6.14 (c) and Chapter III, Corollary 9.3]{Mil06}.
	}
	Then the formula \eqref{0028} holds for $A_{1}$
	if and only if the formula \eqref{0028} holds for $A_{2}$.
\end{theorem}

\begin{proof}
	We need to show that
		\[
				\frac{
					\chi_{W}(X, \mathcal{A}_{2})
				}{
					\chi_{W}(X, \mathcal{A}_{1})
				}
			=
				q^{
						\chi(X, \Lie \mathcal{A}_{2})
					-
						\chi(X, \Lie \mathcal{A}_{1})
				}.
		\]
	Let $\mathcal{A}_{2}'$, $\mathcal{N}_{f}$, $D_{f, v}$ be
	as defined in Section \ref{0018}.
	Since $\mathcal{N}_{f}$ is killed by multiplication by a positive integer,
	we have
	$R \Gamma(X, \mathcal{N}_{f}) \cong R \Gamma_{W}(X, \mathcal{N}_{f})$
	by \eqref{0071}.
	These isomorphic objects are bounded and have finite cohomologies
	by Theorem \ref{0049}.
	Hence,
		$
				\chi(X, \mathcal{N}_{f})
			=
				\chi_{W}(X, \mathcal{N}_{f})
		$.
	With this and \eqref{0070}, we have
		\[
				\chi_{W}(X, \mathcal{A}_{1})
			=
				\chi_{W}(X, \mathcal{N}_{f})
				\chi_{W}(X, \mathcal{A}_{2}')
			=
				\chi(X, \mathcal{N}_{f})
				\chi_{W}(X, \mathcal{A}_{2}').
		\]
	
	The group $D_{f, v}$ for any $v \in X_{0}$ is
	killed by multiplication by a positive integer
	since its identity component is unipotent.
	Hence, $\chi_{W}(k, \Res_{k(v) / k} D_{f, v})$ is well-defined
	and agrees with
		\[
				\chi(k, \Res_{k(v) / k} D_{f, v})
			=
				\chi(k(v), D_{f, v})
		\]
	by the same argument.
	Hence, by \eqref{0029} and \eqref{0070},
	we have
		\[
				\chi_{W}(X, \mathcal{A}_{2})
			=
				\chi_{W}(X, \mathcal{A}_{2}')
				\prod_{v}
					\chi(k(v), D_{f, v}).
		\]
	Therefore,
		\begin{equation} \label{0072}
				\frac{
					\chi_{W}(X, \mathcal{A}_{2})
				}{
					\chi_{W}(X, \mathcal{A}_{1})
				}
			=
				\frac{
					\prod_{v}
						\chi(k(v), D_{f, v})
				}{
					\chi(X, \mathcal{N}_{f})
				}.
		\end{equation}
	
	Arguing similarly, we have
		\begin{equation} \label{0034}
				\chi(X, \Lie \mathcal{A}_{1})
			=
					\chi(X, R \Lie \mathcal{N}_{f})
				+
					\chi(X, \Lie \mathcal{A}_{2}'),
		\end{equation}
		\begin{equation} \label{0035}
				\chi(X, \Lie \mathcal{A}_{2})
			=
					\chi(X, \Lie \mathcal{A}_{2}')
				+
					\sum_{v}
						[k(v) : k]
						\dim_{k(v)}(\Lie D_{f, v}).
		\end{equation}
	Therefore,
		\begin{equation} \label{0073}
					\chi(X, \Lie \mathcal{A}_{2})
				-
					\chi(X, \Lie \mathcal{A}_{1})
			=
					\sum_{v}
						[k(v) : k]
						\dim_{k(v)}(\Lie D_{f, v})
				-
					\chi(X, R \Lie \mathcal{N}_{f}).
		\end{equation}
	
	We have
		\begin{equation} \label{0074}
				\chi(k(v), D_{f, v})
			=
				q_{v}^{\dim D_{f, v}}
			=
				q_{v}^{\dim_{k(v)}(\Lie D_{f, v})}
		\end{equation}
	for each $v$ by Proposition \ref{0005} \eqref{0002}
	(where $q_{v} = \# k(v)$).
	We have
		\begin{equation} \label{0075}
				\chi(X, \mathcal{N}_{f})
			=
				q^{\deg(R \Lie \mathcal{N}_{f})}
			=
				q^{\chi(X, R \Lie \mathcal{N}_{f})}
		\end{equation}
	by Theorem \ref{0030} and Proposition \ref{0043}.
	Combining \eqref{0072}, \eqref{0073}, \eqref{0074} and \eqref{0075},
	we get the desired equality.
\end{proof}


\section{Base change stable invariants}
\label{0078}

The term $\deg(R \Lie \mathcal{N}_{f})$ in Theorem \ref{0031}
changes in a complicated manner
if $K$ is replaced by a finite extension.
We can replace it by a more stable invariant
at the expense of adding and subtracting
base change conductors of the abelian varieties
(Theorem \ref{0047}).
We first need to define this stable invariant in this section.

Let $X$ be a geometrically connected, proper, smooth curve
over a perfect field $k$ of characteristic $p > 0$.
Let $K$ be its function field.
Let $X^{\sep}$ be the normalization of $X$
in the separable closure $K^{\sep}$ of $K$.
Let $\mathcal{F}_{K^{\sep}} \in D^{b}(\Order_{X^{\sep}})$ be
(represented by) a perfect complex of $\Order_{X^{\sep}}$-modules.
Choose a finite Galois extension $L / K$
such that $\mathcal{F}_{K^{\sep}}$ can be obtained as a derived pullback of
a perfect complex $\mathcal{F}_{L}$ of $\Order_{Y}$-modules,
where $Y$ is the normalization of $X$ in $L$.

\begin{definition} \label{0077}
	Define a rational number by
		\[
				\deg_{X}(\mathcal{F}_{K^{\sep}})
			:=
				\frac{
					\deg(\mathcal{F}_{L})
				}{
					[L : K]
				},
		\]
	where the numerator on the right-hand side is the degree
	with respect to $k$
	(see Footnote \ref{0076}).
\end{definition}

Note that the input $\mathcal{F}_{K^{\sep}}$ of the operator $\deg_{X}$
is a perfect complex over $X^{\sep}$,
but its value $\deg_{X}(\mathcal{F}_{K^{\sep}})$ depends not only on $X^{\sep}$
but also on $X$
(hence the subscript $X$ in $\deg_{X}$).

\begin{proposition} \label{0032}
	Definition \ref{0077} is independent of the choice of $L$ or $\mathcal{F}_{L}$.
\end{proposition}

\begin{proof}
	This reduces to the equality
		\[
				\deg(\Order_{X}(v))
			=
				\frac{
					\deg(\Order_{Y}(f^{-1} v))
				}{
					[L : K]
				}
			\quad
				(= [k(v) : k]),
		\]
	where $v$ is a closed point of $X$
	and $f \colon Y \to X$ is the structure morphism.
\end{proof}

On the other hand, if $\mathcal{F}$ is a perfect complex of $\Order_{X}$-modules,
then we denote its base change to $X^{\sep}$ by
$\mathcal{F}_{X^{\sep}}$.
We have $\deg(\mathcal{F}) = \deg_{X}(\mathcal{F}_{X^{\sep}})$
by Proposition \ref{0032}.

We can rewrite base change conductors
using this $\deg_{X}$ construction.
Let $A$ be an abelian variety over $K$
with N\'eron model $\mathcal{A}$ over $X$.
If $A$ has semistable reduction over a finite separable extension $L / K$
and $\mathcal{A}_{L}$ denotes the N\'eron model of $A \times_{K} L$,
then by \cite[Section 7.4, Corollary 4]{BLR90},
$\Lie \mathcal{A}_{L}$ commutes with base change:
	\[
			g^{\ast} \Lie \mathcal{A}_{L}
		\isomto
			\Lie \mathcal{A}_{L'},
	\]
where $L' / L$ is a finite separable extension
and $g \colon Y' \to Y$ the morphism
on the normalizations of $X$ in $L$ and $L'$.
Hence, we may define $\Lie \mathcal{A}_{K^{\sep}}$
to be the base change of $\Lie \mathcal{A}_{L}$ to $X^{\sep}$
for any choice of such an $L$.
It is a locally free sheaf on $X^{\sep}$ of finite rank.

On the other hand, we have the base change
$\Lie \mathcal{A}_{X^{\sep}}$
of $\Lie \mathcal{A}$ to $X^{\sep}$.

\begin{proposition} \label{0045}
	We have
		\[
				c(A)
			=
					\deg_{X}(\Lie \mathcal{A}_{K^{\sep}})
				-
					\deg_{X}(\Lie \mathcal{A}_{X^{\sep}}).
		\]
\end{proposition}

\begin{proof}
	Let $L / K$ be a finite Galois extension
	over which $A$ has semistable reduction everywhere.
	Let $Y$ be the normalization of $X$ in $L$.
	Let $k_{L}$ be the constant field of $Y$.
	We have
		\begin{align*}
			&
					\deg_{X}(\Lie \mathcal{A}_{K^{\sep}})
				-
					\deg_{X}(\Lie \mathcal{A}_{X^{\sep}}).
			\\
			&	=
					\sum_{w \in Y_{0}}
						\frac{
							[k_{L}(w) : k]
						}{
							[L : K]
						}
						\length_{\Order_{w}}
							\frac{
								\Lie \mathcal{A}_{L}
							}{
								\Lie \mathcal{A} \tensor_{\Order_{X}} \Order_{Y}
							}
			\\
			&	=
					\sum_{v \in X_{0}}
						\frac{
							[k(v) : k]
						}{
							e_{v}
						}
						\length_{\Order_{w}}
							\frac{
								\Lie \mathcal{A}_{L}
							}{
								\Lie \mathcal{A} \tensor_{\Order_{X}} \Order_{Y}
							},
		\end{align*}
	where, in the third line,
	$w$ is any point of $Y$ above $v$ that we fix for each $v$
	and $e_{v}$ is the ramification index of $L / K$ at $v$.
	The quantity in the third line is equal to the definition of $c(A)$
	(Definition \ref{0061} \eqref{0060}).
\end{proof}

We stabilize $R \Lie \mathcal{N}_{f}$ as follows.
Let $f \colon A_{1} \to A_{2}$ be an isogeny of abelian varieties over $K$.
Let $L$ be a finite separable extension of $K$
over which $A_{1}$ has semistable reduction.
Let $\mathcal{A}_{i, L}$ be the N\'eron model of $A_{i} \times_{K} L$.
Then the induced morphism
$\Tilde{f}_{L} \colon \mathcal{A}_{1, L} \to \mathcal{A}_{2, L}$
is an isogeny.
Let $\mathcal{N}_{f, L}$ be its kernel.
Then $R \Lie \mathcal{N}_{f, L}$ commutes with base change:
	\[
			g^{\ast} R \Lie \mathcal{N}_{f, L}
		\isomto
			R \Lie \mathcal{N}_{f, L'},
	\]
where $L' / L$ is a finite separable extension
and $g \colon Y' \to Y$ the morphism
of the normalizations of $X$ in $L$ and $L'$.
Hence, we may define $R \Lie \mathcal{N}_{f, K^{\sep}}$ to be
the base change of $R \Lie \mathcal{N}_{f, L}$ to $X^{\sep}$
for any choice of such an $L$.
It is a perfect complex of $\Order_{X^{\sep}}$-modules.


\section{%
	\texorpdfstring{Change in $\mu$ with base change conductors}
	{Change in mu with base change conductors}
}

With these stabilized objects,
we can give a stabilized version of Theorem \ref{0031}:
Let $X$ be a geometrically connected, proper, smooth curve
over a perfect field $k$ of characteristic $p > 0$.
Let $K$ be its function field.

\begin{theorem} \label{0047}
	Let $f \colon A_{1} \to A_{2}$ be an isogeny of abelian varieties over $K$.
	Let $R \Lie \mathcal{N}_{f, K^{\sep}}$ be as defined
	at the end of Section \ref{0078}.
	Then
		\[
				\mu_{A_{2} / K} - \mu_{A_{1} / K}
			=
					\deg_{X}(R \Lie \mathcal{N}_{f, K^{\sep}})
				+ 
					c(A_{2}) - c(A_{1}).
		\]
\end{theorem}

\begin{proof}
	Let $\mathcal{A}_{i}$ be the N\'eron model of $A_{i}$ over $X$.
	Let $\mathcal{N}_{f}$ be the schematic closure of the kernel of $f$ in $\mathcal{A}_{1}$.
	By Theorem \ref{0031}, it is enough to show that
		\[
				\deg(R \Lie \mathcal{N}_{f}) - c(f)
			=
					\deg_{X}(R \Lie \mathcal{N}_{f, K^{\sep}})
				+ 
					c(A_{2}) - c(A_{1}).
		\]
	Here $\deg$ on the left-hand side is the usual degree
	of the determinant invertible sheaf on $X$,
	while $R \Lie \mathcal{N}_{f, K^{\sep}}$ on the right-hand side is
	a perfect complex over $X^{\sep}$
	and $\deg_{X}$ is the operator defined in Definition \ref{0077}.
	We have a morphism of distinguished triangles
		\[
			\begin{CD}
					R \Lie \mathcal{N}_{f, X^{\sep}}
				@>>>
					\Lie \mathcal{A}_{1, X^{\sep}}
				@>>>
					\Lie \mathcal{A}_{2, X^{\sep}}'
				\\ @VVV @VVV @VVV \\
					R \Lie \mathcal{N}_{f, K^{\sep}}
				@>>>
					\Lie \mathcal{A}_{1, K^{\sep}}
				@>>>
					\Lie \mathcal{A}_{2, K^{\sep}}
			\end{CD}
		\]
	of perfect complexes of $\Order_{X^{\sep}}$-modules.
	Taking $\deg_{X}$, we have
		\begin{align*}
			&
						\deg_{X}(R \Lie \mathcal{N}_{f, K^{\sep}})
					-
						\deg_{X}(R \Lie \mathcal{N}_{f, X^{\sep}})
			\\
			&	=
						\deg_{X}(\Lie \mathcal{A}_{1, K^{\sep}})
					-
						\deg_{X}(\Lie \mathcal{A}_{1, X^{\sep}})
					-
						\deg_{X}(\Lie \mathcal{A}_{2, K^{\sep}})
					+
						\deg_{X}(\Lie \mathcal{A}_{2, X^{\sep}}').
		\end{align*}
	By Proposition \ref{0045}, we have
		\[
					\deg_{X}(\Lie \mathcal{A}_{1, K^{\sep}})
				-
					\deg_{X}(\Lie \mathcal{A}_{1, X^{\sep}})
			=
				c(A_{1}).
		\]
	On the other hand,
		\begin{align*}
			&
						\deg_{X}(\Lie \mathcal{A}_{2, K^{\sep}})
					-
						\deg_{X}(\Lie \mathcal{A}_{2, X^{\sep}}')
			\\
			&	=
						c(A_{2})
					+
						\deg_{X}(\Lie \mathcal{A}_{2, X^{\sep}})
					-
						\deg_{X}(\Lie \mathcal{A}_{2, X^{\sep}}').
			\\
			&	=
						c(A_{2})
					+
						c(f).
		\end{align*}
	Combining these, we get the result.
\end{proof}


\section{%
	\texorpdfstring{Another, conditional proof of Theorem \ref{0031} by the $\mu$-invariant formula}
	{Another, conditional proof of Theorem \ref{0031} by the mu-invariant formula}
}

In this section,
we will give another, conditional but simpler proof of Theorem \ref{0031}
using the $\mu$-invariant formula in \cite{LLSTT21}.

Let $X$ be a geometrically connected, proper, smooth curve
over a perfect field $k$ of characteristic $p > 0$.
Let $K$ be its function field.
Assume that $k$ is finite with $q$ elements in this section.
Let $A$ be an abelian variety over $K$.
Write the $L$-function $L(A, t)$ of $A$ as
	\[
			L(A, t)
		=
			\frac{P_{1}(t)}{P_{0}(t) P_{2}(t)},
	\]
where $P_{i}(t) \in \Z[t]$ is a polynomial in $t$
with constant term $1$ whose reciprocal roots are
Weil $q$-numbers of weight $i + 1$.
See \cite[Section 2.2]{LLSTT21}.
Define $\theta_{A}$ to be the nonnegative rational number
such that $q^{\theta_{A}} P_{1}(t / q)$ is $p$-primitive
(namely, all the coefficients are $p$-adically integral and
some coefficients are $p$-adic units).
(It is actually an integer by \cite[Proposition A.3]{LLSTT21}.)
Recall the $\mu$-invariant formula of \cite{LLSTT21}
in the form stated in \cite[Remark 2.5.5]{LLSTT21}:
	\begin{equation} \label{0033}
			\theta_{A}
		=
				\chi^{0}(X, \mathcal{A})
			-
				\chi(X, \Lie \mathcal{A}).
	\end{equation}
This is expected to be true for all $A$.
So far it is verified in the following cases:
\begin{itemize}
	\item
		$\Sha(A / K \F_{q^{p^{n}}})$ is finite for all $n$
		(\cite[Corollary 2.5.1]{LLSTT21}).
	\item
		$A$ is the Jacobian of a geometrically connected, proper, smooth curve over $K$
		(\cite[Corollary 3.4.1]{LLSTT21}).
	\item
		$A$ has semistable reduction everywhere
		(\cite[Theorem 4.1.1]{LLSTT21}).
\end{itemize}

Let $f \colon A_{1} \to A_{2}$ be an isogeny of abelian varieties over $K$.
Assume that $A_{1}$ and $A_{2}$ satisfy the formula \eqref{0033}.
We will give a simpler proof of Theorem \ref{0031} under this assumption.

\begin{proof}[Another proof of Theorem \ref{0031} under \eqref{0033}]
	Let $\mathcal{A}_{i}$ and $\mathcal{N}_{f}$ be as defined in Section \ref{0018}.
	We have $\theta_{A_{1}} = \theta_{A_{2}}$
	since $L$-functions are isogeny invariant.
	By \eqref{0009} and \eqref{0033}, we are reduced to showing that
		\[
					\chi(X, \Lie \mathcal{A}_{1})
				-
					\chi(X, \Lie \mathcal{A}_{2})
			=
				\deg(R \Lie \mathcal{N}_{f}) - c(f).
		\]
	But this follows from \eqref{0034} and \eqref{0035}.
\end{proof}

\begin{remark}
	Without assuming $k$ to be finite,
	we can still define the number $\theta_{A}$
	to be $\sum_{\lambda} (1 - \lambda)$,
	where $\lambda$ runs over the multiset of slopes $< 1$ of
	the $F$-isocrystal $\mathcal{H}^{1}_{\Q_{p}}(-1)$ over $k$
	given in \cite[Appendix]{LLSTT21}
	by \cite[Proposition A.1]{LLSTT21}.
	The formula \eqref{0033} is expected to be true
	for a general $k$ with this definition of $\theta_{A}$.
	It implies Theorem \ref{0031} by the same proof,
	with ``$L$-functions'' replaced by ``$F$-isocrystals.''
\end{remark}


\newcommand{\etalchar}[1]{$^{#1}$}

\end{document}